\documentclass[reqno,centertags]{amsart}

\usepackage{amsmath,amsthm,amscd,amssymb,latexsym,upref,enumerate}
\usepackage{bbm}
\usepackage{enumitem} 
\usepackage{nicefrac}
\usepackage{hyperref}

\newcommand*{\mailto}[1]{\href{mailto:#1}{\nolinkurl{#1}}}
\newcommand{\arxiv}[1]{\href{http://arxiv.org/abs/#1}{arXiv:#1}}
\newcommand{\doi}[1]{\href{http://dx.doi.org/#1}{DOI:#1}}

\newtheorem{theorem}{Theorem}[section]

\newtheorem{lemma}[theorem]{Lemma}
\newtheorem{corollary}[theorem]{Corollary}
\newtheorem{hypothesis}[theorem]{Hypothesis}
\theoremstyle{definition}

\newtheorem{remark}[theorem]{Remark}

\newcommand{\R}{{\mathbb R}}

\newcommand{\C}{{\mathbb C}}

\newcommand{\bbR}{{\mathbb{R}}}

\newcommand{\be}{\begin{equation}}
\newcommand{\ee}{\end{equation}} 
\newcommand{\spr}[2]{\langle #1 , #2 \rangle}
\newcommand{\dbspr}[2]{[ #1 , #2 ]}

\newcommand{\E}{\mathrm{e}}

\newcommand{\im}{\mathrm{Im}}
\newcommand{\re}{\mathrm{Re}}
\newcommand{\dom}[1]{\mathrm{dom}\left(#1\right)}

\newcommand{\no}{\notag}
\newcommand{\lb}{\label}

\newcommand{\ol}{\overline}

\newcommand{\OO}{\mathcal{O}}
\newcommand{\oo}{o}

\newcommand{\bi}{\bibitem}

\DeclareMathOperator{\AC}{AC}
\DeclareMathOperator{\loc}{loc}

\newcommand{\qd}{{[1]}}

\newcommand{\Tmax}{T_{\mathrm{max}}}
\newcommand{\Llocr}{L^1_{\mathrm{loc}}((a,b);r(x)dx)}
\newcommand{\Lr}{L^2((a,b);r(x)dx)}
\newcommand{\Lrmu}{L^2(\R;d\mu)}
\newcommand{\Deftau}{\mathfrak{D}_\tau}

\renewcommand{\Im}{\text{\rm Im}}
\renewcommand{\ln}{\text{\rm ln}}

\newcommand{\vphi}{\varphi}

\newcommand{\lam}{\lambda}

\numberwithin{equation}{section}

\begin{document}

\title[Inverse Spectral Theory]{Inverse Spectral Theory for Sturm--Liouville Operators 
with Distributional Potentials}

\author[J.\ Eckhardt]{Jonathan Eckhardt}
\address{Faculty of Mathematics\\ University of Vienna\\
Nordbergstrasse 15\\ 1090 Wien\\ Austria}
\email{\mailto{jonathan.eckhardt@univie.ac.at}}
\urladdr{\url{http://homepage.univie.ac.at/jonathan.eckhardt/}}

\author[F.\ Gesztesy]{Fritz Gesztesy}
\address{Department of Mathematics,
University of Missouri,
Columbia, MO 65211, USA}
\email{\mailto{gesztesyf@missouri.edu}}
\urladdr{\url{http://www.math.missouri.edu/personnel/faculty/gesztesyf.html}}

\author[R.\ Nichols]{Roger Nichols}
\address{Mathematics Department, The University of Tennessee at Chattanooga, 
415 EMCS Building, Dept. 6956, 615 McCallie Ave, Chattanooga, TN 37403, USA}
\email{\mailto{Roger-Nichols@utc.edu}}
\urladdr{\url{http://rogeranichols.jimdo.com/}} 

\author[G.\ Teschl]{Gerald Teschl}
\address{Faculty of Mathematics\\ University of Vienna\\
Nordbergstrasse 15\\ 1090 Wien\\ Austria\\ and International
Erwin Schr\"odinger
Institute for Mathematical Physics\\ Boltzmanngasse 9\\ 1090 Wien\\ Austria}
\email{\mailto{Gerald.Teschl@univie.ac.at}}
\urladdr{\url{http://www.mat.univie.ac.at/~gerald/}}

\thanks{{\it Research supported by the Austrian Science Fund (FWF) under Grant No.\ Y330}}
\thanks{J.\ Lond.\ Math.\ Soc.\ (2) {\bf 88}, 801--828 (2013).}

\keywords{Sturm--Liouville operators, distributional coefficients, inverse spectral theory.}
\subjclass[2010]{Primary 34B24, 34L05; Secondary  34L40, 46E22.}

%%%%%%%%%%%
\begin{abstract}
We discuss inverse spectral theory for singular differential operators on arbitrary intervals 
$(a,b) \subseteq \mathbb{R}$ associated with rather general differential expressions of the type 
\[
 \tau f = \frac{1}{r} \left( - \big(p[f' + s f]\big)' + s p[f' + s f] + qf\right),  
\]
where the coefficients $p$, $q$, $r$, $s$ are Lebesgue measurable on $(a,b)$ with  $p^{-1}$, $q$, 
$r$, $s \in L^1_{\text{loc}}((a,b); dx)$ and real-valued with $p\not=0$ and $r>0$ a.e.\  
on $(a,b)$. In particular, we explicitly permit certain distributional potential coefficients. 

The inverse spectral theory results derived in this paper include those implied by the spectral measure,  
by two-spectra and three-spectra, as well as local Borg--Marchenko-type inverse spectral results. The 
special cases of Schr\"odinger operators with distributional potentials and Sturm--Liouville operators 
in impedance form are isolated, in particular. 
\end{abstract}
%%%%%%%%%%%

\maketitle

%%%%%%%%%%%%%%%%%%%%%%%%%%%%%%%%%%%%%%%%
%%%%%%%%%%%%%%%%%%%%%%%%%%%%%%%%%%%%%%%%
\section{Introduction}\label{s1}
%%%%%%%%%%%%%%%%%%%%%%%%%%%%%%%%%%%%%%%%
%%%%%%%%%%%%%%%%%%%%%%%%%%%%%%%%%%%%%%%%

The prime motivation behind this paper is to discuss inverse spectral theory for singular Sturm--Liouville operators on an arbitrary interval $(a,b) \subseteq \bbR$ associated with rather general differential expressions of the type 
\be
 \tau f = \frac{1}{r} \left( - \big(p[f' + s f]\big)' + s p[f' + s f] + qf\right).    \lb{1.1}
\ee
Here the coefficients $p$, $q$, $r$, $s$ are Lebesgue measurable on $(a,b)$ with 
 $p^{-1}$, $q$, $r$, $s \in L^1_{\loc}((a,b); dx)$ and real-valued with $p\not=0$ and $r>0$ ae.\ on $(a,b)$.   
Furthermore, $f$ is supposed to satisfy
\begin{equation} 
f \in AC_{\text{loc}}((a,b)), \; p[f' + s f] \in AC_{\text{loc}}((a,b)), 
\end{equation}
with $AC_{\loc}((a,b))$ denoting the set of all locally absolutely continuous functions on $(a,b)$. For such functions $f$, the expression 
\begin{equation}
f^{[1]} := p[f'+ s f]     \lb{1.2} 
\end{equation} 
will subsequently be called the {\it first quasi-derivative} of $f$.

One notes that in the general case \eqref{1.1}, the differential expression is formally given by
\be
\tau f = \frac{1}{r} \left( - \big(pf'\big)' + \big[- (sp)' + s^2 p + q\big]f \right). 	  \lb{1.3}
\ee
In the special case $s= 0$ this approach reduces to the standard one, that is, one obtains
\be
\tau f = \frac{1}{r} \left( - \big(pf'\big)' + q f \right).
\ee
Moreover, in the case $p=r=1$, our approach is sufficiently general to include distributional potential coefficients from the space 
$W^{-1,1}_{\loc}((a,b))$ as well as all of $H^{-1}_{\loc}((a,b)) = W^{-1,2}_{\loc}((a,b))$ (as the term 
$s^2$ can be absorbed in $q$), and thus even in this special case our setup is slightly more general than the approach pioneered by Savchuk and Shkalikov \cite{SS99}, who defined the differential expression as
\be
\tau f = - \big([f' + s f]\big)' + s [f' + s f] - s^2 f, \quad f,\, [f' + s f] \in \AC_{\loc}((a,b)). 
\ee
One observes that in this case $q$ can be absorbed in $s$ by virtue of the transformation $s \to s+ \int^x q$. Their approach requires the additional condition $s^2 \in L^1_{\loc}((a,b); dx)$.  Moreover, since there are distributions in $H^{-1}_{\loc}((a,b))$ which are not measures, the operators discussed here are not a special case of Sturm--Liouville operators with measure-valued coefficients as discussed, for instance, in \cite{BR05}, \cite{ET12}.

As we pointed out in our extensive introductions in the recent papers \cite{EGNT12} and \cite{EGNT12a} 
on the subject of general Sturm--Liouville operators with distributional coefficients, 
similar differential expressions have previously been studied by Bennewitz and Everitt \cite{BE83} in 
1983 (see also \cite[Section I.2]{EM99}). Moreover, an extremely thorough and systematic investigation, including even and odd higher-order operators defined in terms of appropriate quasi-derivatives, and 
in the general case of matrix-valued coefficients (including distributional potential coefficients in the 
context of Schr\"odinger-type operators) was presented by Weidmann \cite{We87} in 1987.
However, it seems likely to us that the extraordinary generality exerted by Weidmann \cite{We87} in his treatment of higher-order differential operators obscured the fact that he already dealt with distributional potential coefficients back in 1987 and so it was not until 1999 that Savchuk and Shkalikov \cite{SS99} 
started a new development for Sturm--Liouville (resp., Schr\"odinger) operators with distributional 
potential coefficients in connection with areas such as, self-adjointness proofs, spectral and inverse 
spectral theory, oscillation properties, spectral properties in the non-self-adjoint context, etc. 

Next, rather than describing the enormous amount of literature that followed the work by 
Savchuk and Shkalikov \cite{SS99} (such a description has been undertaken in great detail in 
\cite{EGNT12} and \cite{EGNT12a}), and before describing the content of this paper, we now focus 
exclusively on prior work that is directly related to inverse spectral aspects involving distributional 
coefficients. First we note that measures, as positive distributions of order zero, represent of course a 
special case of distributions and we refer, for instance, to Ben Amor and Remling \cite{BR05}. 

Much of the existing literature deals with the case where $p=r=1$, $s$ is given in terms of the 
integral of an element (possibly, complex-valued) in $W^{\alpha,2}((0,1))$, $\alpha \geq -1$, 
and separated Dirichlet and Robin-type boundary conditions are formulated (the latter in terms 
of the first quasi-derivative \eqref{1.2}). For instance, Hryniv and Mykytyuk \cite{HM03} studied 
the inverse spectral problem for Dirichlet boundary conditions in the self-adjoint case for 
$\alpha = -1$, using one spectrum and an appropriate set of norming constants. Their results 
include the reconstruction algorithm and its stability, and a description of isospectral sets. The 
case of the two spectra inverse problem (fixing the boundary condition at one end) is also discussed 
in \cite{HM04aa} and continued in \cite{AHM08}, \cite{Hr11}, \cite{HM06}, 
\cite{HM06a}. In addition, Robin-type boundary conditions and a Hochstadt--Lieberman-type 
inverse spectral result, where the knowledge of the set of norming constants is replaced by knowledge 
of the potential over the interval $(0,1/2)$, and again a reconstruction algorithm is provided in 
\cite{HM04}; transformation operators associated with Robin boundary conditions are treated in 
\cite{HM04a}. The case of the Sobolev scale, that is, the case $\alpha \in [-1,0]$ is developed in 
\cite{HM06} (see also \cite{HM06a}), and analyticity and uniform stability of these inverse spectral 
problems in this Sobolev scale are established in Hryniv \cite{Hr11}. For the case of self-adjoint 
matrix-valued distributional potentials we refer to Mykytyuk and Trush \cite{MT10}; the case of 
reconstruction by three spectra is treated in Hryniv and Mykytyuk \cite{HM03a}. A discontinuous 
inverse eigenvalue problem, permitting a jump in the eigenfunctions at an interior point of the 
underlying interval, extending a result of Hochstadt--Lieberman, was proved by Hald \cite{Ha84}. 

In the complex-valued case, Savchuk and Shkalikov studied the inverse spectral 
problem, in particular, the reconstruction problem from two spectra in \cite{SS05}. Moreover, in 
\cite{SS06}--\cite{SS10} they discussed the case of Dirichlet boundary conditions 
and also the case with 
Dirichlet boundary conditions at one end and Neumann boundary conditions at the other in connection 
with the inverse spectral problem in terms of one spectrum and (appropriate analogs) of norming 
constants and obtain interesting mapping properties associated with (regularized) spectral data 
that lead to a solution of the problem of uniform stability of recovering the potential. 
In this context we also mention the paper by Albeverio, Hryniv, and Mykytyuk \cite{AHM08} 
which focuses on discrete spectra of non-self-adjoint Schr\"odinger  operators with (complex-valued) distributional potentials $W^{-1,2} ((0,1))$ (i.e., in the special case where $p=r=1$, 
$s\in L^2((0,1);dx)$) and either Dirichlet boundary conditions at both endpoints, or Dirichlet boundary conditions at one endpoint and Robin-type boundary conditions at the other. In particular, the 
reconstruction of the potential from two spectra, or one spectrum and appropriate (analogs of) 
norming constants are provided. 

In contrast to these inverse spectral problems associated with discrete spectrum situations 
on a bounded interval, inverse scattering problems for Schr\"odinger operators with $p=r=1$ 
and distributional potentials $q$ represented in terms of of the Miura transformation 
$q = s' + s^2$ for $s \in L^1(\bbR; dx) \cap L^2(\bbR; dx)$ are treated in 
Frayer, Hryniv, Mykytyuk, and Perry \cite{FHMP09}. These studies are continued and extended 
in Hryniv, Mykytyuk, and Perry \cite{HMP11, HMP11a}; in this context we also refer to the basic 
paper by Kappeler, Perry, Shubin, and Topalov \cite{KPST05} on the Miura transform on the real 
line. 

Next, we turn to a description of the principal content of this paper. In Section \ref{s2} we briefly 
discuss the relevant basics of Sturm--Liouville operators associated with the differential expression 
\eqref{1.1}, as far as they are needed in this 
paper (precise details with proofs can be found in \cite{EGNT12}, \cite{EGNT12a}). Inverse spectral 
results for these kind of operators based on the spectral measure are derived in Section \ref{s3} 
employing the theory of 
de Branges spaces. The first result of this type (assuming $r=1$, $s=0$ and a regular left endpoint) 
appears to be due to Bennewitz \cite{Be03}. He proved that, in general, the spectral measure determines 
the operator only up to a so-called Liouville transform by employing certain Paley--Wiener-type results.
Since the Paley--Wiener spaces are a particular kind of de Branges spaces, his approach is quite close to ours. In fact, instead of the explicit Paley--Wiener theorems in \cite{Be03}, we here use the more abstract 
framework of de Branges' theory which also allows us to handle quite singular left endpoints. 
In addition, let us mention that similar problems for left-definite 
Sturm--Liouville operators have been considered quite recently in \cite{Be04}--\cite{BBW12}, 
\cite{E12a} due to their relevance in connection with solving the Camassa--Holm 
equation. However, our principal result of Section \ref{s3}, Theorem \ref{thmdBuniqS}, also yields 
relevant results in the special case of Schr\"odinger operators with distributional potentials 
(assuming $p=r=1$) and in the context of Sturm--Liouville operators in impedance form (assuming 
$q=s=0$ and $p=r$). Section \ref{s4} is then devoted to the inverse spectral problem associated 
with two discrete spectra, a classical situation especially, in the case of Schr\"odinger operators. In our principal result of this section, Theorem \ref{thmTS}, we extend the two-spectrum result now to the 
general situation \eqref{1.1}, and hence to the case of distributional coefficients. Moreover, we are 
also able to permit the left endpoint to be in the limit circle (l.c.) case instead of the usual assumption of 
being regular. Again, the cases of Schr\"odinger and 
impedance-type Sturm--Liouville operators are isolated as important special cases. The corresponding 
case of three spectra, associated with the intervals $(a,b)$, $(a,c)$, and $(c,b)$ for some fixed 
$c \in (a,b)$ is then treated in Section \ref{s5}, again with particular emphasis on the special cases 
of Schr\"odinger and impedance-type Sturm--Liouville operators. Our final Section \ref{s6} then covers 
the case of local Borg--Marchenko as well as Hochstadt--Lieberman results for Schr\"odinger operators 
with distributional potential coefficients.
In particular, we generalize the main results of \cite{ET12a}, \cite{KST12} to the distributional situation. 
Appendix \ref{sA} derives some high-energy asymptotic 
relations for regular Weyl--Titchmarsh functions under our general set of assumptions but assuming 
that $\tau$ is regular at the left endpoint $a$; Appendix \ref{sB}  derives a more refined high-energy asymptotics of regular Weyl--Titchmarsh functions in the special case $p=r=1$, again 
assuming $\tau$ to be regular at $a$.   

Finally, we briefly mention some of the notation used in this paper: The Hilbert spaces are typically of 
the form $\Lr$ with scalar product denoted by $\spr{\,\cdot\,}{\cdot\,}_{r}$ (linear in the first factor). 
In addition, we denote by $L^2_0((a,c);r(x)dx)$ the closed linear subspace of $\Lr$ consisting of all 
functions which vanish (Lebesgue) a.e.\ outside of $(a,c)$. 
Given a linear operator $T$ mapping (a subspace of) a Hilbert space into another, $\dom{T}$, 
$\sigma(T)$, $\rho(T)$ denote the domain, spectrum, and resolvent set of $T$.

%%%%%%%%%%%%%%%%%%%%%%%%%%%%%%%%%%%%%%%%
%%%%%%%%%%%%%%%%%%%%%%%%%%%%%%%%%%%%%%%%
\section{Sturm--Liouville Operators with Distributional Coefficients} \lb{s2}
%%%%%%%%%%%%%%%%%%%%%%%%%%%%%%%%%%%%%%%%
%%%%%%%%%%%%%%%%%%%%%%%%%%%%%%%%%%%%%%%%

In the present section we recall the basics of Sturm--Liouville operators with distributional potential coefficients as far as it is needed for the present paper. Throughout this section we make the following assumptions:

%%%%%%%%%%%%
\begin{hypothesis} \lb{h2.1}
Suppose $(a,b) \subseteq \bbR$ and assume that 
$p$, $q$, $r$, $s$ are Lebesgue measurable on $(a,b)$ with $p^{-1}$, $q$, $r$, $s\in L^1_{\loc}((a,b);dx)$ and real-valued with $p\not=0$ and $r>0$ $($Lebesgue\,$)$ a.e.\ on $(a,b)$.
\end{hypothesis}
%%%%%%%%%%%%

The differential expressions $\tau$ considered in this paper are of the form
\be
 \tau f = \frac{1}{r} \left( - \big(p [f' + s f]\big)' + s p [f' + s f] + qf\right),\quad f\in \Deftau.  \lb{2.2}
\ee
Here, $\Deftau$ is the set of all functions $f$ for which $\tau f$ belongs to $\Llocr$,   
\begin{equation}
\Deftau=\{g\in AC_{\text{loc}}((a,b))\, |\, p [g' + s g] \in AC_{\text{loc}}((a,b))\}.
\end{equation}
 Instead of the usual derivative, we will use the expression 
\begin{equation}
f^{[1]}=p [f' + s f], \quad f \in \Deftau,    \lb{2.3}
\end{equation}
which will be referred to as the {\it first quasi-derivative} of $f$.

Given some $g\in\Llocr$, the equation $(\tau-z) f=g$ for some $z\in\C$ is equivalent to the system of ordinary differential equations
\be\label{eqn:system}
 \begin{pmatrix} f \\ f^\qd \end{pmatrix}' 
  = \begin{pmatrix}  - s & p^{-1} \\ q - zr & s  \end{pmatrix}  \begin{pmatrix} f \\ f^\qd  \end{pmatrix} - \begin{pmatrix} 0 \\ rg \end{pmatrix}
\ee
on $(a,b)$. 
As a consequence, for each $c\in(a,b)$ and $d_1$, $d_2\in\C$ there is a unique solution  $f\in\Deftau$ of  $(\tau - z) f = g$ with $f(c)=d_1$ and $f^\qd(c)=d_2$. Moreover, if $g$, $d_1$, $d_2$, and $z$ are real-valued, then so is the solution $f$.

For all functions $f$, $g\in\Deftau$ the modified Wronski determinant
\be
 W(f,g)(x) = f(x)g^\qd(x) - f^\qd(x)g(x), \quad x\in (a,b),
\ee
is locally absolutely continuous. In fact, one has the Lagrange identity 
\be\label{eqn:lagrange}
 \int_\alpha^\beta \left[g(x) (\tau f)(x)  - f(x) (\tau g)(x)\right] \, r(x) dx = W(f,g)(\beta) - W(f,g)(\alpha)
\ee
for each  $\alpha$, $\beta\in(a,b)$. 
 As a consequence, the Wronskian $W(u_1,u_2)$ of two solutions $u_1$, $u_2\in\Deftau$ of $(\tau-z)u=0$ is constant with $W(u_1,u_2)  \not=0$ if and only if $u_1$,  $u_2$ are linearly independent.

 Associated with the differential expression $\tau$ is a maximally defined linear operator $\Tmax$ in the Hilbert space $\Lr$ with scalar product
\be
 \spr{f}{g}_{r} = \int_{a}^b  f(x) \ol{g(x)} \, r(x) dx, \quad f,\, g\in\Lr,
\ee
which is given by 
\begin{align} 
& \Tmax f = \tau f,   \\
& f \in \dom{\Tmax} = \left\lbrace g\in\Lr \,\big|\, g\in\Deftau, \, \tau g\in\Lr \right\rbrace.  \no 
\end{align}
 It turns out that this operator always has self-adjoint restrictions and their description depends on whether 
 the limit circle (l.c.) or limit point (l.p.) cases prevail at the endpoints, which are described as follows.  
 The differential expression $\tau$ is said to be in the l.c.\ case at an endpoint if for one (and hence 
 for all) $z\in\C$, all solutions of $(\tau-z)u=0$ lie in $L^2((a,b);r(x)dx)$ near this endpoint. Otherwise, 
 $\tau$ is said to be in the l.p.\ case at this endpoint. 
 Moreover, $\tau$ is called {\it regular} at an endpoint, if $p^{-1}$, $q$, $r$, $s$ are integrable near this endpoint. In this case, for each solution $u$ of $(\tau-z)u=0$, the functions $u(x)$ and $u^\qd(x)$ 
 have finite limits as $x$ tends to this endpoint. This guarantees that $\tau$ is always in the l.c.\ case 
 at regular endpoints. 

 In this paper we are only interested in self-adjoint restrictions with separated boundary conditions. 
 All of them are given by 
\begin{align}
\begin{split} 
 & S f = \tau f,   \\ 
 & f \in \dom{S} = \left\lbrace g\in\dom{\Tmax} \,|\, W(g,\ol{v})(a) = W(g,\ol{w})(b) = 0 \right\rbrace,  
 \end{split} 
\end{align}
for some suitable $v$, $w\in\dom{\Tmax}$. 
Hereby, the boundary condition at an endpoint is actually only necessary if $\tau$ is in the l.c.\ case at this endpoint (otherwise, i.e., in the l.p.\ case, it is void).  
In the l.c.\ case one needs to require that $W(v,\ol{v})(a) = 0$ and $W(h,\ol{v})(a)\not=0$, respectively $W(w,\ol{w})(b) = 0$ and $W(h,\ol{w})(b)\not=0$ for 
some $h\in\dom{\Tmax}$. We also note that all self-adjoint restrictions of $\Tmax$ have separated boundary conditions unless the l.c.\ case prevails at both endpoints. 

 Next, let $S$ be some self-adjoint restriction of $\Tmax$ with separated boundary conditions. 
 In order to define a singular Weyl--Titchmarsh--Kodaira function (as introduced recently in \cite{ET12}, 
 \cite{GZ06}, and \cite{KST12}), one needs a real entire fundamental system $\theta_z$, $\phi_z$ 
 of $(\tau-z)u=0$ with $W(\theta_z,\phi_z)=1$, such that $\phi_z$ lies in $\dom{S}$ near $a$, 
 that is, $\phi_z$ lies in $\Lr$ near $a$ and satisfies the boundary condition at $a$ if $\tau$ is in the 
 l.c.\ case at $a$.
 Here, by {\it real entire} it is meant that the functions 
 \begin{align}
   z\mapsto\theta_z(c), \quad z\mapsto\theta_z^\qd(c), \quad z\mapsto \phi_z(c), 
   \quad z\mapsto \phi_z^\qd(c), 
 \end{align}
are real entire for some (and hence for all) $c\in(a,b)$. For such a real entire fundamental system to 
exist, it is necessary and sufficient that there is no essential spectrum coming from the singularity at 
$a$; see \cite[Theorem\ 8.6]{EGNT12a}, \cite{GZ06}, \cite{KST12}. In order to make this precise, we 
fix some $c\in(a,b)$ and let $S_{(a,c)}$ be some self-adjoint operator in $L^2((a,c);r(x)dx)$ associated 
to $\tau$ restricted to the interval $(a,c)$. 

%%%%%%%%%%
\begin{theorem}\label{thmREFS}
 The following items are equivalent:
 \begin{enumerate}[label=\,\emph{(}\roman*\emph{)}, leftmargin=*, widest=iii, align=left]
  \item There exists a real entire fundamental system $\theta_z$, $\phi_z$ of $(\tau-z)u=0$ with $W(\theta_z,\phi_z)=1$, such that $\phi_z$ lies in $\dom{S}$ near $a$.
  \item There is a non-trivial real entire solution $\phi_z$ of $(\tau-z)u=0$ which lies in $\dom{S}$ near $a$.
  \item The spectrum of $S_{(a,c)}$ is purely discrete for some $c\in(a,b)$. 
 \end{enumerate}
\end{theorem}
%%%%%%%%%%

Given a real entire fundamental system as in Theorem \ref{thmREFS}, one can define a function \,$m:\rho(S)\rightarrow\C$\, by requiring that the solutions
\be\label{defpsi}
\psi_z = \theta_z + m(z)\phi_z, \quad z\in\rho(S),
\ee
lie in $\dom{S}$ near $b$, that is, they lie in $\Lr$ near $b$ and satisfy the boundary condition at 
$b$, if $\tau$ is 
 in the l.c.\ case at $b$.
 This function $m$, called the singular Weyl--Titchmarsh--Kodaira function of $S$, is analytic on 
 $\rho(S)$ and satisfies 
 \begin{align}\label{eqprop::mpsi} 
  m(z)= \ol{m(\ol{z})}, \quad z\in\rho(S).
 \end{align}
 By virtue of the Stieltjes--Livsic inversion formula, it is possible to associate a measure with $m$. In 
 fact, there is a unique Borel measure $\mu$ on $\R$ given by
\be\label{defrho}
 \mu((\lambda_1,\lambda_2]) = \lim_{\delta\downarrow 0}\,\lim_{\varepsilon\downarrow 0} \frac{1}{\pi} 
                      \int_{\lambda_1+\delta}^{\lambda_2+\delta} \im(m(\lambda+i\varepsilon))\, d\lambda,
\ee
 for each $\lambda_1$, $\lambda_2\in\R$ with $\lambda_1<\lambda_2$. The transformation 
 \be
  \hat{f}(z) = \int_a^b \phi_z(x) f(x) \, r(x) dx, \quad z\in\C, \lb{eqnhat}
 \ee
 initially defined for functions $f\in\Lr$ which vanish near $b$, uniquely extends to a unitary operator $\mathcal{F}$ from $\Lr$ onto $\Lrmu$, which maps $S$ to multiplication  with the independent variable in $\Lrmu$. 
 Since the measure $\mu$ is uniquely determined by this property, it is referred to as the spectral measure of $S$ associated with the real entire solution $\phi_z$. 
 
%%%%%%%%%%
\begin{remark}\label{remSFS}
 It is worth noting that a real entire fundamental systems of $(\tau-z)u=0$ as in Theorem \ref{thmREFS} is not unique and any other such system is given by   
 \begin{align}
  \tilde{\theta}_z = \E^{-g(z)} \theta_z - f(z)\phi_z, \quad \tilde{\phi}_z 
  = \E^{g(z)} \phi_z, \quad z\in\C,     \lb{2.15} 
 \end{align}
 for some real entire functions $f$, $\E^g$. The corresponding singular Weyl--Titchmarsh--Kodaira functions are related via
 \begin{align}
  \widetilde{m}(z) = \E^{-2g(z)} m(z) + \E^{-g(z)}f(z), \quad z\in\rho(S).
 \end{align}
 In particular, the maximal domain of holomorphy and the structure of poles and singularities do not change.
 Moreover, the corresponding spectral measures are related by
\begin{align}
 d\tilde{\mu} = \E^{-2g} d\mu.
\end{align}
 Hence, they are mutually absolutely continuous and the associated spectral transformations only differ by a simple rescaling with the positive function $\E^{-2g}$. One also notes that the spectral measure does not depend on the function $f$ in \eqref{2.15}.
\end{remark}
%%%%%%%%%%

%%%%%%%%%%%%%%%%%%%%%%%%%%%%%%%%%%%%
%%%%%%%%%%%%%%%%%%%%%%%%%%%%%%%%%%%%
\section{Inverse Uniqueness Results in Terms of the Spectral Measure}\label{s3} 
%%%%%%%%%%%%%%%%%%%%%%%%%%%%%%%%%%%%
%%%%%%%%%%%%%%%%%%%%%%%%%%%%%%%%%%%%

In this section we will use the theory of de Branges spaces in order to show to which extent the spectral measure determines the coefficients (and boundary conditions) of the Sturm--Liouville operator in question. 
 For an exposition of this theory we refer to de Branges' book \cite{dB68} or to \cite{Dy70}, 
 \cite{DMcK76}, \cite{Re02}, in which particular emphasis is placed on applications to Sturm--Liouville operators.
 As in the previous section, let $S$ be some self-adjoint realization  of $\tau$, which satisfies Hypothesis \ref{h2.1}, with separated boundary conditions, such that there is a real entire solution $\phi_z$ of $(\tau-z)u=0$ which lies in $\dom{S}$ near $a$. 
 The spectral measure corresponding to this solution will be denoted by $\mu$. 
 In order to introduce the de Branges spaces associated with $S$ and the real entire solution $\phi_z$, we fix some point $c\in(a,b)$ and consider the entire function
\begin{align}\label{eqndBschrE}
 E(z,c) = \phi_z(c) + i \phi_z^\qd(c), \quad z\in\C.
\end{align}
Using the Lagrange identity and the fact that the Wronskian of two solutions satisfying the same boundary condition at $a$ (if any) vanishes at $a$, one gets
\begin{align}\label{eqnELag} 
 \frac{E(z,c) \overline{E(\zeta,c)} - E(\overline{\zeta},c) \overline{E(\overline{z},c)}}{2i (\overline{\zeta} -z)} = \int_a^c \phi_z(x) \overline{\phi_\zeta(x)} \, r(x) dx, \quad \zeta,\,z\in\C_+,
\end{align}  
 where $\C_+$ denotes the open upper complex half-plane.
 In particular, taking $\zeta=z$, this shows that $E(\,\cdot\,,c)$ is a de Branges function, that is, 
 $|E(z,c)|>|E(\overline{z},c)|$ for every $z\in\C_+$.
Moreover, one notes that $E(\,\cdot\,,c)$ does not have any real zero $\lambda$, since otherwise 
$\phi_\lambda$, as well as its quasi-derivative, would vanish at $c$.
With $\mathfrak{B}(c)$ we denote the de Branges space (see, e.g., \cite[Section\ 19]{dB68}, 
\cite[Section\ 2]{Re02}) associated with the de Branges function $E(\,\cdot\,,c)$. 
It consists of all entire functions $F$ such that
\begin{align}
 \int_\R \frac{|F(\lambda)|^2}{|E(\lambda,c)|^2} \, d\lambda < \infty
\end{align}
and such that $F/E(\,\cdot\,,c)$ and $F^\#/E(\,\cdot\,,c)$ are of bounded type in $\C_+$ (that is, they can be written as the quotient of two bounded analytic functions) with non-positive mean type. 
Here $F^\#$ denotes the entire function  
\begin{align}
 F^\#(z) = \overline{F(\overline{z})}, \quad z\in\C, 
\end{align}
 and the mean type of a function $N$ which is of bounded type in $\C_+$ is the number
 \begin{align}
  \limsup_{y\rightarrow\infty} \frac{\ln|N(i y)|}{y} \in[-\infty,\infty).
\end{align}
 Endowed with the inner product
\begin{align}
 \dbspr{F}{G}_{\mathfrak{B}(c)} = \frac{1}{\pi} \int_\R \frac{F(\lambda) 
 \overline{G(\lambda)}}{|E(\lambda,c)|^2} \, d\lambda, 
  \quad F,\, G\in \mathfrak{B}(c),
\end{align}
 the space $\mathfrak{B}(c)$ turns into a reproducing kernel Hilbert space with
 \begin{align}
  F(\zeta) = \dbspr{F}{K(\zeta,\cdot\,,c)}_{\mathfrak{B}(c)}, \quad F\in \mathfrak{B}(c),
 \end{align}
 for each $\zeta\in\C$. 
 Here, the reproducing kernel $K(\,\cdot\,,\cdot\,,c)$ is given by (see also  \cite[Theorem\ 19]{dB68}, \cite[Section\ 3]{Re02}) 
\begin{align}\label{eqndBschrRepKer}
 K(\zeta,z,c) = \int_a^c \phi_z(x) \overline{\phi_\zeta(x)} \, r(x)dx, \quad \zeta,\,z\in\C,
\end{align} 
 as a similar calculation as the one which led to \eqref{eqnELag} shows.
 In order to state the next result, we recall our convention that $L^2_0((a,c);r(x)dx)$ denotes the closed linear subspace of $\Lr$ consisting of all functions which vanish (Lebesgue) a.e.\ outside of $(a,c)$. 

%%%%%%%%%%%%
\begin{theorem}\label{thmdBschrBT}
 For every $c\in(a,b)$ the transformation $f\mapsto\hat{f}$ is unitary from $L^2_0((a,c);r(x)dx)$ 
 onto $\mathfrak{B}(c)$, in particular
 \begin{align}
  \mathfrak{B}(c) = \big\lbrace  \hat{f} \,\big|\, f\in L^2_0((a,c);r(x)dx) \big\rbrace.
 \end{align}
\end{theorem}
 %%%%%%%%%%%%
\begin{proof}
 For each $\lambda\in\R$, the transform of the function $f_\lambda$ defined by 
 \begin{align}
  f_\lambda(x) = \begin{cases} \phi_\lambda(x), & x\in(a,c], \\
                               0,               & x\in(c,b), \end{cases}
 \end{align}
 is given by
 \begin{align}
  \hat{f}_\lambda(z) = \int_a^c \phi_\lambda(x) \phi_z(x) \, r(x)dx = K(\lambda,z,c), \quad z\in\C.
 \end{align}
 In particular, this guarantees that the transforms of all the functions $f_\lambda$, $\lambda\in\R$, 
 lie in $\mathfrak{B}(c)$. Moreover, one has for every $\lambda_1$, $\lambda_2\in\R$, that 
 \begin{align}
 \begin{split}
  \spr{f_{\lambda_1}}{f_{\lambda_2}}_r & = \int_a^c \phi_{\lambda_1}(x) \phi_{\lambda_2}(x) \, r(x)dx    \\ 
  &   = K(\lambda_1,\lambda_2,c) = \dbspr{K(\lambda_1,\cdot\,,c)}{K(\lambda_2,\cdot\,,c)}_{\mathfrak{B}(c)}. 
 \end{split} 
 \end{align}
 Thus, the transform $f\mapsto\hat{f}$ is an isometry from the linear span $D$ of all functions 
 $f_\lambda$, $\lambda\in\R$, to $\mathfrak{B}(c)$. But this span is clearly dense in $L^2_0((a,c);r(x)dx)$ since it contains all eigenfunctions of some self-adjoint operator $S_{(a,c)}$ (associated with 
 $\tau|_{(a,c)}$ in $L^2((a,c);r(x)dx)$).
 Moreover, the linear span of all transforms $K(\lambda,\cdot\,,c)$, $\lambda\in\R$, is dense in 
 $\mathfrak{B}(c)$. In fact, each $F\in \mathfrak{B}(c)$ with 
 \begin{align}
  0 = \dbspr{F}{K(\lambda,\cdot\,,c)}_{\mathfrak{B}(c)} = F(\lambda), \quad \lambda\in\R,
 \end{align}
 vanishes identically. Hence, the transformation $f\mapsto\hat{f}$ restricted to $D$ uniquely extends to a unitary map $V$ from $L^2_0((a,c);r(x)dx)$ onto $\mathfrak{B}(c)$.
 In order to identify this unitary map with the transformation $f\mapsto\hat{f}$, one notes that for each 
 given $z\in\C$, both $f\mapsto\hat{f}(z)$ and $f\mapsto (Vf)(z)$ are continuous on $L^2_0((a,c);r(x)dx)$.
\end{proof}
%%%%%%%%%%%%

As an immediate consequence of Theorem \ref{thmdBschrBT} and the fact that the transformation 
\eqref{eqnhat} extends to a unitary map from $\Lr$ onto $\Lrmu$, one obtains the following corollary.

%%%%%%%%%%%%
\begin{corollary}\label{cordBschrembL2}
For every $c\in(a,b)$, the de Branges space $\mathfrak{B}(c)$ is isometrically embedded in $\Lrmu$, that is,
\begin{align}
 \int_\R |F(\lambda)|^2 d\mu(\lambda) = \|F\|^2_{\mathfrak{B}(c)}, \quad F\in \mathfrak{B}(c).
\end{align}
Moreover, the union of all spaces $\mathfrak{B}(c)$, $c\in(a,b)$, is dense in $\Lrmu$, that is, 
\begin{align}\label{eqndBschrdense}
 \overline{\bigcup_{c\in(a,b)} \mathfrak{B}(c)} = \Lrmu.
\end{align}
\end{corollary}
%%%%%%%%%%%%

Theorem \ref{thmdBschrBT} also implies that the de Branges spaces $\mathfrak{B}(c)$, $c\in(a,b)$ form an ascending chain of subspaces of $\Lrmu$ which is continuous in the sense of \eqref{eqndBschrcontinuous} 
below. 

%%%%%%%%%%%%
\begin{corollary}\label{cordBschrincl}
If $c_1$, $c_2\in(a,b)$ with $c_1<c_2$, then $\mathfrak{B}(c_1)$ is isometrically embedded in, but not equal to, $\mathfrak{B}(c_2)$.
Moreover,
\begin{align}\label{eqndBschrcontinuous}
 \overline{\bigcup_{x\in(a,c)} \mathfrak{B}(x)} = \mathfrak{B}(c) = \bigcap_{x\in(c,b)}\mathfrak{B}(x), \quad c\in (a,b).
\end{align}
\end{corollary}
%%%%%%%%%%%%
\begin{proof}
 The embedding is immediate from Theorem \ref{thmdBschrBT} and Corollary \ref{cordBschrembL2}. 
  Moreover, from Theorem \ref{thmdBschrBT} one infers that $\mathfrak{B}(c_2)\,\ominus\, \mathfrak{B}(c_1)$ is unitarily equivalent to the space $L^2((c_1,c_2);r(x)dx)$, hence $\mathfrak{B}(c_1)$ is not equal to $\mathfrak{B}(c_2)$.
  The remaining claim follows from the analogous fact that
  \begin{align}
   \overline{\bigcup_{x\in(a,c)} L^2_0((a,x);r(x)dx)} = L^2_0((a,c);r(x)dx) 
   = \bigcap_{x\in(c,b)} L^2_0((a,x);r(x)dx).
  \end{align}
\end{proof}
%%%%%%%%%%%%

Next, we proceed to the main result of this section; an inverse uniqueness result. For this purpose, let 
$\tau_j$, $j=1,2$, be two Sturm--Liouville differential expressions on the interval $(a_j,b_j)$, $j=1,2$, of the form \eqref{2.2}, both satisfying the assumptions made in Hypothesis \ref{h2.1}. With $S_j$, $j=1,2$, we denote two corresponding self-adjoint realizations with separated boundary conditions, and we suppose 
that there are nontrivial real entire solutions which lie in the domain of the respective operator near the left endpoint $a_j$, $j=1,2$. All quantities corresponding to $S_1$ and $S_2$ are denoted in an obvious way 
with an additional subscript. In order to state the following result, one recalls that an analytic function is said to be of {\it bounded type} if it can be written as the quotient of two bounded analytic functions. 

%%%%%%%%%%%%
\begin{theorem}\label{thmdBuniqS}
  Suppose that the function 
\begin{align}\label{eqnquotE1E2}
   \frac{E_1(z,x_1)}{E_2(z,x_2)}, \quad z\in\C_+,
\end{align}
is of bounded type for some $x_1\in (a_1,b_1)$ and $x_2\in(a_2,b_2)$. If $\mu_1= \mu_2$, then there is a locally absolutely continuous bijection $\eta$ from $(a_1,b_1)$ onto $(a_2,b_2)$, and locally absolutely continuous, real-valued functions $\kappa$, $\nu$ on $(a_1,b_1)$ such that 
 \begin{align}
 \begin{split} 
   \eta' r_2\circ\eta & = \kappa^{2} r_1, \lb{3.15} \\
   p_2\circ\eta & = \eta' \kappa^{2} p_1, \\
   \eta' s_2\circ\eta & = s_1 + \kappa^{-1}(\kappa' - \nu p_1^{-1}), \\
   \eta' q_2\circ\eta & = \kappa^2 q_1 + 2\kappa\nu s_1 - \nu^2 p_1^{-1} + \kappa'\nu - \kappa\nu'.
 \end{split} 
 \end{align}
 Moreover, the map $V$ given by 
 \begin{align}
V:  \begin{cases} 
 L^2((a_2,b_2);r_2(x)dx)\rightarrow L^2((a_1,b_1);r_1(x)dx),  \\
 f_2 \mapsto (V f_2)(x_1) = \kappa(x_1) f_2(\eta(x_1)), \quad x_1\in(a_1,b_1), 
 \end{cases} 
 \end{align}
is unitary, with 
\begin{equation} 
S_1 = V S_2 V^{-1}.
\end{equation}    
\end{theorem}
%%%%%%%%%%%%
\begin{proof}
First of all one notes that the function in \eqref{eqnquotE1E2} is of bounded type for all points 
$x_j\in(a_j,b_j)$, $j=1, 2$ (see, e.g., \cite[Lemma\ 2.2]{E12}), and we now fix some arbitrary 
$x_1\in(a_1,b_1)$. Since for each $x_2\in(a_2,b_2)$, both $\mathfrak{B}_1(x_1)$ and 
$\mathfrak{B}_2(x_2)$ are 
isometrically contained in $L^2(\R;d\mu_1)$, one infers from de Branges' subspace ordering 
theorem (see \cite[Theorem\ 35]{dB68}), that either $\mathfrak{B}_1(x_1)$ is contained in 
$\mathfrak{B}_2(x_2)$ or $\mathfrak{B}_2(x_2)$ is contained in $\mathfrak{B}_1(x_1)$.
Next, we claim that the infimum $\eta(x_1)$ of all $x_2\in(a_2,b_2)$ for which 
$\mathfrak{B}_1(x_1)\subseteq \mathfrak{B}_2(x_2)$, lies in $(a_2,b_2)$.
 Otherwise, either $\mathfrak{B}_2(x_2)\subseteq \mathfrak{B}_1(x_1)$ for all $x_2\in(a_2,b_2)$, 
 or $\mathfrak{B}_1(x_1)\subseteq \mathfrak{B}_2(x_2)$ for all $x_2\in(a_2,b_2)$.
 In the first case, one infers that $\mathfrak{B}_1(x_1)$ is dense in $L^2(\R;d\mu_1)$, which is impossible in view of Corollary \ref{cordBschrincl}.
  In the second case, one obtains for each function $F\in \mathfrak{B}_1(x_1)$ and $\zeta\in\C$ that 
 \begin{align}
  |F(\zeta)|^2 & \leq \left| \dbspr{F}{K_2(\zeta,\cdot\,,x_2)}_{\mathfrak{B}_2(x_2)} \right|^2
 \leq \|F\|_{\mathfrak{B}_2(x_2)}^2 \dbspr{K_2(\zeta,\cdot\,,x_2)}{K_2(\zeta,\cdot\,,x_2)}_{\mathfrak{B}_2(x_2)} 
 \no \\
& = \|F\|_{\mathfrak{B}_1(x_1)}^2 K_2(\zeta,\zeta,x_2), \quad x_2\in (a_2,b_2).
 \end{align}
However, \eqref{eqndBschrRepKer} implies $K_2(\zeta,\zeta,x_2)\rightarrow0$ as $x_2\rightarrow a_2$ and, as a result, $\mathfrak{B}_1(x_1)=\lbrace 0\rbrace$, which contradicts Theorem \ref{thmdBschrBT}.
From \eqref{eqndBschrcontinuous} one infers that 
 \begin{align}
   \mathfrak{B}_2(\eta(x_1)) = \overline{\bigcup_{x_2<\eta(x_1)} \mathfrak{B}_2(x_2)} \subseteq \mathfrak{B}_1(x_1) \subseteq \bigcap_{x_2>\eta(x_1)} \mathfrak{B}_2(x_2) = \mathfrak{B}_2(\eta(x_1)),
 \end{align}
 and hence $\mathfrak{B}_1(x_1)=\mathfrak{B}_2(\eta(x_1))$, including equality of norms. 
 In particular, this implies the existence of functions $\kappa$, $\nu$ on $(a_1,b_1)$ with
 \begin{align}\label{eqndBdet}
  \begin{pmatrix} \phi_{1,z}(x_1) \\ \phi_{1,z}^\qd(x_1) \end{pmatrix} 
  = \begin{pmatrix} \kappa(x_1) & 0 \\ \nu(x_1) & \kappa(x_1)^{-1} \end{pmatrix} 
  \begin{pmatrix} \phi_{2,z}(\eta(x_1)) \\ \phi_{2,z}^\qd(\eta(x_1)) \end{pmatrix}, 
  \quad z\in\C, \; x_1\in(a_1,b_1),
 \end{align}
in view of \cite[Theorem\ I]{dB60} and the high-energy asymptotics in Theorem \ref{A.2}, which imply 
$\phi_{j,z}(x_j) \phi_{j,z}^\qd(x_j)^{-1} \rightarrow 0$ as $|z|\rightarrow\infty$ along non-real rays. 
 
The function $\eta: (a_1,b_1)\rightarrow(a_2,b_2)$ defined above is strictly increasing in view of Corollary \ref{cordBschrincl}, and continuous by \eqref{eqndBschrcontinuous}.
Moreover, since for each  $\zeta\in\C$,
\begin{align}
    K_2(\zeta,\zeta,\eta(x_1)) = K_1(\zeta,\zeta,x_1) \rightarrow 0 \quad \text{as}\quad 
    x_1\rightarrow a,
\end{align}
one concludes that $\eta(x_1)\rightarrow a_2$ as $x_1\rightarrow a_1$.
Furthermore, \eqref{eqndBschrdense} shows that $\eta$ actually has to be a bijection.  
Using the equation for the reproducing kernels \eqref{eqndBschrRepKer} once more, one obtains for 
each $z\in\C$ that 
\begin{align}
  \int_{a_1}^{x_1} |\phi_{1,z}(x)|^2 r_1(x) dx = \int_{a_2}^{\eta(x_1)} |\phi_{2,z}(x)|^2 r_2(x) dx, \quad x_1\in(a_1,b_1).
\end{align}
This fact and an application of \cite[Chapter\ IX; Exercise\ 13]{Na55} and 
\cite[Chapter\ IX; Theorem\ 3.5]{Na55} imply that $\eta$ is locally absolutely continuous. 
 Hence, in view of \eqref{eqndBdet}, $\kappa$, $\kappa^{-1}$, and $\nu$ are locally absolutely continuous. 
One notes that the functions $\phi_{2,z}\circ\eta$ and $\phi_{2,z}^\qd\circ\eta$ are locally absolutely continuous by \cite[Chapter\ IX; Theorem\ 3.5]{Na55} since $\eta$ is strictly increasing. 
Differentiating equation \eqref{eqndBdet} with respect to $x_1\in(a_1,b_1)$ and using the differential 
equation (more precisely, the equivalent system \eqref{eqn:system}), yields the four relations in 
\eqref{3.15}. 

 Finally, one considers the unitary operator $\mathcal{F}_1^{-1} \mathcal{F}_2$ from $L^2((a_2,b_2);r_2(x)dx)$ onto $L^2((a_1,b_1);r_1(x)dx)$ for which one has 
 $S_1 = \mathcal{F}_1^{-1} \mathcal{F}_2 S_2 \mathcal{F}_2^{-1} \mathcal{F}_1$. 
 In order to identify this operator with $V$, one notes that for each $F\in L^2(\R;d\mu_1)$,
 \begin{align}
  \mathcal{F}_1^{-1}F(x_1) = \kappa(x_1) \int_\R \phi_{2,\lambda}(\eta(x_1)) F(\lambda) d\mu_1(\lambda) = \kappa(x_1) \mathcal{F}_2^{-1} F(\eta(x_1)),
 \end{align}
for a.e.\ $x_1\in (a_1,b_1)$.  This concludes the proof.
 \end{proof}
 %%%%%%%%%%%%

We note that the function defined in \eqref{eqnquotE1E2} of Theorem \ref{thmdBuniqS} is of bounded type if the de Branges functions $E_j(\,\cdot\,,x_j)$, $j=1, 2$,  for some $x_j\in(a_j,b_j)$, satisfy a particular growth condition. In fact, one can suppose that they belong to the Cartwright class, that is, that they are of finite exponential type and the logarithmic integrals
\begin{align}
 \int_\R \frac{\ln_+(|E_j(\lambda,x_j)|)}{1+\lambda^2} \, d\lambda < \infty, \quad j=1, 2,    \lb{3.24} 
\end{align}
exist (here $\ln_+$ denotes the positive part of the natural logarithm). 
Then, a theorem of Krein \cite[Theorem\ 6.17]{RR94}, \cite[Section\ 16.1]{Le96} states that these functions are of bounded type in the open upper and lower complex half-plane, and hence the quotient \eqref{eqnquotE1E2} in Theorem \ref{thmdBuniqS} is of bounded type as well. 
 In particular, we will show in the following section that it is always possible to choose particular real entire solutions $\phi_{1,z}$, $\phi_{2,z}$ such that \eqref{3.24} holds, provided the l.c.\ case prevails at the left endpoint. We emphasize that condition \eqref{eqnquotE1E2} in Theorem \ref{thmdBuniqS} cannot be dropped and that some additional assumption has to be imposed on the de Branges functions. Indeed, a counterexample in \cite{E12} shows that dropping such an additional condition is not even possible 
 in the standard case of regular Schr\"odinger operators. 
 
Naturally, if one has {\it a priori} additional information on some of the coefficients, one ends up with stronger uniqueness results. In particular, we are interested in the special case of Schr\"odinger operators with a distributional potential (i.e., the special case $p = r = 1$), for which the inverse uniqueness result simplifies considerably. 

%%%%%%%%%%%%
\begin{corollary}\label{corInvUniqSchr}
Suppose that $p_j=r_j=1$, $j=1, 2$, and that the function 
  \begin{align}
   \frac{E_1(z,x_1)}{E_2(z,x_2)}, \quad z\in\C_+, 
  \end{align}
  is of bounded type for some $x_1\in (a_1,b_1)$ and $x_2\in(a_2,b_2)$.
  If $\mu_1= \mu_2$, then there is some $\eta_0\in\R$ and a locally absolutely continuous, real-valued function $\nu$ on $(a_1,b_1)$ such that 
  \begin{align}
  \begin{split} 
  s_2(\eta_0 +x_1) & = s_1(x_1) + \nu(x_1), \\
  q_2(\eta_0 +x_1) & = q_1(x_1) -2\nu(x_1)s_1(x_1) - \nu(x_1)^2 +\nu'(x_1),
  \end{split} 
\end{align}
 for almost all $x_1\in(a_1,b_1)$.  Moreover, the map $V$ given by 
 \begin{align}
V:  \begin{cases} 
 L^2((a_2,b_2);dx)\rightarrow L^2((a_1,b_1);dx),  \\
 f_2 \mapsto (V f_2)(x_1) = f_2(\eta_0 + x_1), \quad x_1\in(a_1,b_1), 
 \end{cases} 
 \end{align}
is unitary, with 
\begin{equation} 
S_1 = V S_2 V^{-1}.
\end{equation}    
\end{corollary}
%%%%%%%%%%%%

We note that the claim in Corollary \ref{corInvUniqSchr} formally implies (cf.\ \eqref{1.3}) 
 \begin{align}
  -s_2'(\eta_0 +x_1) + s_2(\eta_0 +x_1)^2 + q_2(\eta_0 +x_1) = -s_1'(x_1) + s_1(x_1)^2 + q_1(x_1).
 \end{align}
 Thus, the distributional potential of such a Schr\"odinger operator is uniquely determined by the spectral measure up to a shift by $\eta_0$. 
 More precisely, it is quite obvious that the transformation $V$ which maps a function $f_2$ on $(a_2,b_2)$ onto the shifted function $Vf_2$ given by 
 \begin{align}
  (V f_2)(x_1) = f_2(\eta_0+x_1), \quad x_1\in(a_1,b_1),
 \end{align}
 takes the domain $\mathfrak{D}_{\tau_2}$ onto  $\mathfrak{D}_{\tau_1}$.    
 Furthermore, a straightforward calculation shows that for each $f_2\in\mathfrak{D}_{\tau_2}$ we 
 have $(\tau_2 f_2)(\eta_0 +x_1) = (\tau_1 V f_2)(x_1)$ for almost all $x_1\in(a_1,b_1)$, which says 
 that $\tau_1$ and $\tau_2$ act the same up to a shift by $\eta_0$. 
 We will elaborate on this particular case of Schr\"odinger operators with distributional potentials in 
 Section \ref{s6}. 
 
 Finally, our inverse uniqueness theorem also allows us to deduce a result for Sturm--Liouville operators 
 in {\it impedance form}, that is, in the special case $q=s=0$ and $p=r$. This special case received 
 particular attention in the literature and we refer, for instance, to \cite{AHM05}, \cite{An88}, \cite{An88a}, 
\cite[Section XVII.4]{CS89}, \cite{CM93}, \cite{Gl91}, \cite{HM03}, \cite{RS92}, and the pertinent literature cited therein. 

%%%%%%%%%%%%% 
\begin{corollary}\label{corInvUniqSchrImpedance}
Suppose that $q_j=s_j=0$, $p_j=r_j$, $j=1, 2$, and that the function 
\begin{align}
\frac{E_1(z,x_1)}{E_2(z,x_2)}, \quad z\in\C_+, 
\end{align}
is of bounded type for some $x_1\in (a_1,b_1)$ and $x_2\in(a_2,b_2)$.
If $\mu_1= \mu_2$, then there is a $c_1\in(a_1,b_1)$ and constants $\eta_0$, $\nu_0$, $\kappa_0\in\R$ such that 
\begin{align}
p_2(\eta_0 +x_1) & = p_1(x_1) \left(\nu_0 \int_{c_1}^{x_1} \frac{dt}{p_1(t)} + \kappa_0\right)^2 
\end{align}
for almost all $x_1\in(a_1,b_1)$. Moreover, the map $V$ given by  
\begin{align}
V: \begin{cases} 
L^2((a_2,b_2);r_2(x)dx) \to L^2((a_1,b_1);r_1(x)dx), \\
f_2 \mapsto (V f_2)(x_1) = \left(\nu_0 \int_{c_1}^{x_1} \frac{dt}{p_1(t)} + \kappa_0\right) f_2(\eta_0 + x_1), \quad x_1\in(a_1,b_1),
\end{cases} 
\end{align}
is unitary, with 
\begin{equation}
S_1 = V S_2 V^{-1}.
\end{equation} 
\end{corollary}
%%%%%%%%%%%%%
 
 %%%%%%%%%%%%%
 \begin{remark} \lb{r3.7} 
 We note that the function  
 \begin{align}
  \nu_0 \int_{c_1}^{x_1} \frac{dt}{p_1(t)} + \kappa_0, \quad x_1\in(a_1,b_1), 
 \end{align}
 in Corollary \ref{corInvUniqSchrImpedance} resembles the function $\kappa$ in 
 Theorem \ref{thmdBuniqS} which must not vanish anywhere. Consequently, the constants appearing in Corollary \ref{corInvUniqSchrImpedance} are restricted to values such that this function has no zeros. 
This allows one to strengthen our inverse uniqueness result, provided that one imposes some additional assumptions on the coefficients: For example, if it is known {\it a priori} that $p_1^{-1}$ is neither integrable near the left nor near the right endpoint, then one concludes that $\nu_0$ is zero and hence the impedance functions are equal up to a scalar multiple (and a simple shift by $\eta_0$). 
 Moreover, if one assumes that the left endpoints are regular with the same non-Dirichlet boundary conditions there, and that the real entire solutions $\phi_{1,z}$ and $\phi_{2,z}$ have the same boundary values there, then it is also possible to conclude that $\nu_0$ is zero (in view of \eqref{eqndBdet}). 
 Finally, the constant $\nu_0$ also has to be zero if one assumes that both differential expressions are 
 regular at one endpoint with Neumann boundary conditions there (upon letting $x_1$ tend to this 
 endpoint in \eqref{eqndBdet}). 
As a last remark, we note that because of the impedance form, when $\nu_0$ is zero, both differential expressions and hence also both operators will act the same (up to a shift by $\eta_0$) although the impedance functions are only equal up to a scalar multiple (and a shift by $\eta_0$). 
\end{remark}
%%%%%%%%%%%%%

%%%%%%%%%%%%%%%%%%%%%%%%%%%%%%%%%%%%%%%%
%%%%%%%%%%%%%%%%%%%%%%%%%%%%%%%%%%%%%%%%
 \section{Inverse Uniqueness Results in Terms of Two Discrete Spectra} \lb{s4} 
%%%%%%%%%%%%%%%%%%%%%%%%%%%%%%%%%%%%%%%%
%%%%%%%%%%%%%%%%%%%%%%%%%%%%%%%%%%%%%%%%

 As in the previous section, let $\tau$ be a Sturm--Liouville expression of the form \eqref{2.2}, satisfying Hypothesis \ref{h2.1}, and $S$ be some self-adjoint realization with separated boundary conditions. We will now show that instead of the spectral measure, it also suffices to know two discrete spectra associated with two different boundary conditions at the left endpoint. Necessarily, therefore, this endpoint has to be in the l.c.\ case. In this context there exists a particular choice of real entire fundamental systems as in Theorem \ref{thmREFS} (see \cite[Section\ 8]{EGNT12a}).
 
 %%%%%%%%%%%%
 \begin{theorem}\label{thmLC}
 If $\tau$ is in the l.c.\ case at $a$, then there exists a real entire fundamental system  $\theta_z$, $\phi_z$ of $(\tau-z)u=0$ with $W(\theta_z,\phi_z)=1$, such that $\phi_z$ lies in $\dom{S}$ near $a$, 
 \begin{align}
  W\big(\theta_{z_{{}_1}},\phi_{z_{{}_2}}\big)(a) = 1, \quad 
  W\big(\theta_{z_{{}_1}},\theta_{z_{{}_2}}\big)(a) = W\big(\phi_{z_{{}_1}},\phi_{z_{{}_2}}\big)(a)=0, \quad z_1,\,z_2\in\C.
 \end{align}
 In this case, the corresponding singular Weyl--Titchmarsh--Kodaira function $m$ is a Nevanlinna--Herglotz function with representation 
 \begin{align}\label{eqnSThergMrep}
  m(z) = \re(m(i)) + \int_\R 
 \bigg( \frac{1}{\lambda-z} - \frac{\lambda}{1+\lambda^2} \bigg) d\mu(\lambda),\quad z\in\C\backslash\R.
 \end{align}
 \end{theorem}
 %%%%%%%%%%%%
 
 In contrast to the general case (cf.\ Remark \ref{remSFS}), real entire fundamental systems $\theta_z$, $\phi_z$ as in Theorem \ref{thmLC} are unique up to scalar multiples and $\theta_z$ is unique up to adding scalar multiples of $\phi_z$.  More precisely, all fundamental systems of solutions with the properties of Theorem \ref{thmLC} are given by
 \begin{align}\label{eqnLCnewFS}
  \tilde{\theta}_z = \E^{-g_0} \theta_z - f_0 \phi_z, \quad  \tilde{\phi}_z = \E^{g_0} \phi_z, \quad z\in\C,
 \end{align}
 for some constants $f_0\in\R$ and $g_0$ real modulo $i\pi$. The corresponding singular Weyl--Titchmarsh--Kodaira functions are related by 
 \begin{align}\label{eqnLCnewM}
 \widetilde{m}(z) = \E^{-2g_0} m(z) + \E^{-g_0} f_0, \quad z\in\rho(S),
 \end{align}
 and the spectral measures via  
 \begin{align}
 d\tilde{\mu} = \E^{-2g_0} d\mu.
 \end{align}
 For example, if $\tau$ is regular at $a$, then one can take $\theta_z$, $\phi_z$ to be the fundamental system of solutions of $(\tau-z)u=0$ with the initial conditions 
 \begin{equation} 
\theta_z(a)=\phi_z^\qd(a)=\cos(\varphi_a), \quad -\theta_z^\qd(a)=\phi_z(a)=\sin(\varphi_a),   
\end{equation} 
for some suitable $\varphi_a\in[0,\pi)$.
 Since these solutions clearly are of exponential growth order less than one (cf.\ \cite[Theorem\ 2.7]{EGNT12a}), they belong to the Cartwright class. 
 The following lemma shows that this is the case as well for general solutions as in Theorem \ref{thmLC}. 
We note that this result (for classical Sturm--Liouville operators on the half-line) is essentially contained in an article by Krein \cite{Kr52}. 

%%%%%%%%%%%%
 \begin{lemma}\label{lemLC}
  Suppose that $\tau$ is in the l.c.\ case at $a$ and $\theta_z$, $\phi_z$ is a fundamental system as in Theorem \ref{thmLC}. Then the real entire functions 
  \begin{align}\label{eqnLCef}
    z\mapsto\theta_z(c), \quad z\mapsto\theta_z^\qd(c), \quad z\mapsto \phi_z(c), \quad 
     z\mapsto \phi_z^\qd(c), 
  \end{align}
  belong to the Cartwright class for each $c\in(a,b)$. 
 \end{lemma}
%%%%%%%%%%%%
 \begin{proof}
  First of all one notes that for each fixed $c\in(a,b)$,
  \begin{align}
   \frac{1}{\theta_z(c)^2} = \frac{\phi_z(c)}{\theta_z(c)} \left(\frac{\phi_z^\qd(c)}{\phi_z(c)} 
   - \frac{\theta_z^\qd(c)}{\theta_z(c)}\right), \quad z\in\C\backslash\R.
  \end{align} 
  Since all fractions on the right-hand side are Nevanlinna--Herglotz functions (up to a sign), they 
  are bounded by 
  \begin{align}
   C \frac{1+|z|^2}{|\im(z)|} \leq K \exp\left(M\frac{1+\sqrt{|z|}}{\sqrt[4]{|\im(z)|}}\right), 
   \quad z\in\C\backslash\R,
  \end{align}
  for some constants $C$, $K$, $M\in (0,\infty)$. Consequently, one has 
  \begin{align}
   \left|\frac{1}{\theta_z(c)^2}\right| \leq 2 K^2 \exp\left(2 M\frac{1+\sqrt{|z|}}{\sqrt[4]{|\im(z)|}}\right), \quad z\in\C\backslash\R,
  \end{align}
and hence a theorem by Matsaev \cite[Theorem\ 26.4.4]{Le96} implies that the first function in 
\eqref{eqnLCef} belongs to the Cartwright class. The claim for the remaining functions is proved similarly. 
 \end{proof}
 %%%%%%%%%%%%
  
 \begin{remark} 
 We note in passing that Lemma \ref{lemLC} permits certain conclusions on the eigenvalue distribution of particular Sturm--Liouville operators. 
 In fact, if $\tau$ is in the l.c.\ case at both endpoints, then the eigenvalues of $S$ are the zeros of an entire function which belongs to the Cartwright class. 
 Consequently, the eigenvalues have convergence exponent at most one, that is, 
 \begin{align}
   \inf\biggr\lbrace \omega\geq0 \,\bigg|\, \sum_{\lambda\in \sigma(S)} \frac{1}{1+|\lambda|^\omega}
   < \infty \biggr\rbrace \leq 1. 
 \end{align}
 Moreover, if $S$ is semi-bounded from above or from below, then one even has
 \begin{align}
  \sum_{\lambda\in\sigma(S)} \frac{1}{1+|\lambda|} < \infty,
 \end{align}
 in view of \cite[Theorem\ 17.2.1]{Le96}. 
 \end{remark}  
   
 As a consequence of Lemma \ref{lemLC}, the de Branges functions associated with a real entire 
 solution as in Theorem \ref{thmLC} belong to the Cartwright class as well. 
 Thus, as noted in the remark after Theorem \ref{thmdBuniqS}, they are of bounded type in the open upper and lower complex half-plane.
 In particular, the quotient \eqref{eqnquotE1E2} will be of bounded type provided $\phi_{1,z}$ and 
 $\phi_{2,z}$ are real entire solutions as in Theorem \ref{thmLC}. 
 We are now able to deduce a two-spectra inverse uniqueness result from Theorem \ref{thmdBuniqS}. 
 Therefore, let $\tau_j$, $j=1,2$, again be two Sturm--Liouville differential expressions of the 
 form \eqref{2.2}, both satisfying the assumptions made in Hypothesis \ref{h2.1}. 
 
%%%%%%%%%%%%
 \begin{theorem}\label{thmTS}
Suppose that $\tau_j$ are in the l.c.\ case at $a_j$ and that $S_{j}$, $T_{j}$ are two distinct self-adjoint realizations of $\tau_j$ with discrete spectra and the same boundary condition at $b_j$ 
$($if any$)$, $j=1, 2$. If 
  \begin{align}\label{eqnTSequal}
  \sigma(S_{1}) = \sigma(S_{2}) \quad\text{and}\quad \sigma(T_{1}) = \sigma(T_{2}),
\end{align}
then there is a locally absolutely continuous bijection $\eta$ from $(a_1,b_1)$ onto $(a_2,b_2)$, 
and locally absolutely continuous, real-valued functions $\kappa$, $\nu$ on $(a_1,b_1)$ such that 
  \begin{align}
  \begin{split} 
 \eta' r_2\circ\eta & = \kappa^{2} r_1, \\
  p_2\circ\eta & = \eta' \kappa^{2} p_1, \\
  \eta' s_2\circ\eta & = s_1 + \kappa^{-1}(\kappa' - \nu p_1^{-1}), \\
 \eta' q_2\circ\eta & = \kappa^2 q_1 + 2\kappa\nu s_1 - \nu^2 p_1^{-1} + \kappa'\nu - \kappa\nu'.
 \end{split} 
\end{align}
 Moreover, the map $V$ given by 
 \begin{align}\label{eqnLiouvilleT}
 V: \begin{cases}
 L^2((a_2,b_2);r_2(x)dx) \to L^2((a_1,b_1);r_1(x)dx), \\
 f_2 \mapsto (V f_2)(x_1) = \kappa(x_1) f_2(\eta(x_1)), \quad x_1\in(a_1,b_1),
  \end{cases}
 \end{align}
is unitary, with  
\begin{equation} 
S_{1} = V S_2 V^{-1}  \quad\text{and}\quad T_1 = V T_2 V^{-1}.   \lb{4.16}
\end{equation} 
\end{theorem}
%%%%%%%%%%%%
 \begin{proof}
Let $\theta_{j,z}$, $\phi_{j,z}$ be real entire fundamental systems as in Theorem \ref{thmLC} 
associated with the operators $S_{j}$, $j=1, 2$. In particular, the set of poles of the corresponding singular Weyl--Titchmarsh--Kodaira function $m_j$ is precisely the spectrum of $S_{j}$, $j=1,2$. 
 Furthermore, note that the value $h_j = m_j(\lambda)$ is independent of $\lambda\in\sigma(T_j)$ 
and conversely, each $\lambda\in\C$ for which $h_j = m_j(\lambda)$  is an eigenvalue of $T_j$. 
  Hence, one may assume (upon replacing $\theta_{j,z}$ by $\theta_{j,z} + h_j\phi_{j,z}$, cf.\ 
  \eqref{eqnLCnewM}) that the set of zeros of $m_j$ is precisely the 
  spectrum of $T_{j}$.  As a consequence of our assumptions \eqref{eqnTSequal} and a theorem of 
  Krein \cite[Theorem\ 27.2.1]{Le96}, one obtains that $m_1=C^2 m_2$ for some positive constant $C$ and hence, upon replacing $\phi_{1,z}$ with $C \phi_{1,z}$ (cf.\ \eqref{eqnLCnewM}), one can assume that 
  $m_1 = m_2$. Thus we also have $\mu_1= \mu_2$ and hence, except for the very last part in 
  \eqref{eqnLiouvilleT} and \eqref{4.16}, the claim 
  follows from Theorem \ref{thmdBuniqS} in view of Lemma \ref{lemLC}. 
  In order to show that $T_2$ is mapped into $T_1$ via $V$, it suffices to verify that 
  \begin{align}\label{eqntheta}
   \theta_{1,z}(x_1) = \kappa(x_1) \theta_{2,z}(\eta(x_1)), \quad x_1\in(a_1,b_1), \; z\in\C.
  \end{align}
It is a straightforward calculation to check that the right-hand side in this equation is a real entire 
solution of $(\tau_1-z)u=0$. Moreover, since the transformation \eqref{eqnLiouvilleT} leaves the 
Wronskian of two functions invariant, this solution satisfies the properties in Theorem \ref{thmLC} 
(with $\phi_{1,z}$ as the second real entire solution). Consequently, one concludes that 
$\theta_{1,z} = \kappa\, \theta_{2,z}\circ\eta - f_0 \phi_{1,z}$ for some $f_0\in\R$. However, for each eigenvalue $\lambda\in\sigma(T_2)$, the solution $\kappa\, \theta_{2,\lambda}\circ\eta$ lies in $L^2((a_1,b_1);r_1(x)dx)$ near $b_1$ and satisfies the boundary condition of $T_1$ at $b_1$ (if any), which 
is the same as that of $S_1$ by the first part of the proof. This concludes the proof since it allows one 
to infer that $f_0=0$.   
 \end{proof}
%%%%%%%%%%%%

%%%%%%%%%%%%
 \begin{remark}
 At this point, we comment on how a Liouville transform affects the boundary condition at a regular endpoint. 
 Therefore, in addition to the assumptions of Theorem \ref{thmTS}, we now suppose that $\tau_1$ and 
 $\tau_2$ are regular at the left endpoint and that the boundary condition of $S_1$ there is given by 
 \begin{align}
  f(a_1) \cos(\varphi_{a,1}) - f^\qd(a_1) \sin(\varphi_{a,1}) = 0 
 \end{align}
 for some $\varphi_{a,1}\in[0,\pi)$. Then a direct calculation, using, for example, \eqref{eqndBdet}, shows that the boundary condition of $S_2$ at the left endpoint is given by
 \begin{align}
  f(a_2) \left(\kappa(a_1) \cos(\varphi_{a,1}) - \nu(a_1) \sin(\varphi_{a,1})\right) - f^\qd(a_2) \kappa(a_1)^{-1} \sin(\varphi_{a,1}) = 0.
 \end{align}
In this context one notes that the limits of $\kappa(x_1)$, $\nu(x_1)$ and $\kappa(x_1)^{-1}$ as $x_1\rightarrow a_1$ exist, which can be verified upon considering \eqref{eqndBdet} and \eqref{eqntheta}. 
 \end{remark}
 %%%%%%%%%%%%

   Uniqueness results in inverse spectral theory in connection with two discrete spectra are intimately connected with classical results for Schr\"odinger operators due to Borg, Levinson, Levitan, and 
Marchenko. Within the limited scope of this paper it 
is impossible to provide an exhaustive list of all pertinent references. Hence, we can only limit ourselves listing a few classical results, such as, \cite{Bo46}, \cite{Bo52}, \cite{Le49}, \cite{Le68}, 
\cite[Chapter 3]{Le87}, \cite{LG64}, \cite{Ma73}, \cite[Section 3.4]{Ma11}, and some of the more recent ones in connection with distributional potential coefficients, such as, \cite{AHM08}, \cite{Hr11}, \cite{HM03}, 
\cite{HM04aa}, \cite{HM06}, \cite{HM06a}, \cite{SS05}--\cite{SS10} (some of these references also derive reconstruction algorithms for the potential term in Schr\"odinger operators); the interested reader will find many more references in these sources.
 
 As in Section \ref{s3} (cf.\ Corollary \ref{corInvUniqSchr}), our inverse uniqueness result simplifies 
 in the case of Schr\"odinger operators with a distributional potential.

%%%%%%%%%%%%
\begin{corollary}
Suppose that $p_j = r_j = 1$, that $\tau_j$ are in the l.c.\ case at $a_j$, and that $S_{j}$, $T_{j}$ 
are two distinct self-adjoint realizations of $\tau_j$ with discrete spectra and the same boundary condition 
at $b_j$ $($if any$)$, $j=1,2$. If 
  \begin{align}
  \sigma(S_{1}) = \sigma(S_{2}) \quad\text{and}\quad \sigma(T_{1}) = \sigma(T_{2}),
  \end{align}
  then there is some $\eta_0\in\R$ and a locally absolutely continuous, real-valued function $\nu$ on $(a_1,b_1)$ such that 
  \begin{align}
  \begin{split}
  s_2(\eta_0 +x_1) & = s_1(x_1) + \nu(x_1), \\
  q_2(\eta_0 +x_1) & = q_1(x_1) -2\nu(x_1)s_1(x_1) - \nu(x_1)^2 +\nu'(x_1),
  \end{split} 
\end{align}
 for almost all $x_1\in(a_1,b_1)$.  Moreover, the map $V$ given by 
 \begin{align}
V:  \begin{cases} 
 L^2((a_2,b_2);dx)\rightarrow L^2((a_1,b_1);dx),  \\
 f_2 \mapsto (V f_2)(x_1) = f_2(\eta_0 + x_1), \quad x_1\in(a_1,b_1), 
 \end{cases} 
 \end{align}
is unitary, with 
\begin{equation} 
S_{1} = V S_2 V^{-1}  \quad\text{and}\quad T_1 = V T_2 V^{-1}.
\end{equation} 
\end{corollary}
%%%%%%%%%%%%
 
We conclude this section by stating the following result for Sturm--Liouville operators in impedance form as in Section \ref{s3} (cf.\ Corollary \ref{corInvUniqSchrImpedance}). 

%%%%%%%%%%%%% 
\begin{corollary}\label{corTSimped}
Suppose that $q_j = s_j = 0$, $p_j=r_j$, that $\tau_j$ are in the l.c.\ case at $a_j$, and that 
$S_{j}$, $T_{j}$ are two distinct self-adjoint realizations of $\tau_j$ with discrete spectra and the same boundary condition at $b_j$  $($if any$)$, $j=1, 2$. If 
\begin{align}
\sigma(S_{1}) = \sigma(S_{2}) \quad\text{and}\quad \sigma(T_{1}) = \sigma(T_{2}),
\end{align}
then there is a $c_1\in(a_1,b_1)$ and constants $\eta_0$, $\nu_0$, $\kappa_0\in\R$ such that 
\begin{align}
p_2(\eta_0 +x_1) & = p_1(x_1) \left(\nu_0 \int_{c_1}^{x_1} \frac{dt}{p_1(t)} + \kappa_0\right)^2 
\end{align}
for almost all $x_1\in(a_1,b_1)$. Moreover, the map $V$ given by  
\begin{align}
V: \begin{cases} 
L^2((a_2,b_2);r_2(x)dx) \to L^2((a_1,b_1);r_1(x)dx),   \\
f_2 \mapsto (V f_2)(x_1) = \left(\nu_0 \int_{c_1}^{x_1} \frac{dt}{p_1(t)} + \kappa_0\right) f_2(\eta_0 + x_1), 
\quad x_1\in(a_1,b_1),
\end{cases} 
\end{align}
is unitary, with 
\begin{equation} 
S_1 = V S_2 V^{-1} \quad\text{and}\quad T_1 = V T_2 V^{-1}.
\end{equation} 
\end{corollary}
 %%%%%%%%%%%%%%
 
 At first sight, this result might seem to be weaker than, for example, the uniqueness part of 
 \cite[Theorem 7.1]{AHM05}, where the impedance function is determined by two spectra up to a 
 scalar multiple. However, this is not the case since in \cite{AHM05} a particular choice of two spectra 
 with Dirichlet and Neumann boundary conditions is used. As already mentioned in the remark after 
 Corollary \ref{corInvUniqSchrImpedance}, this additional knowledge suffices to conclude that $\nu_0$ 
 is zero and hence the impedance functions are equal up to a scalar multiple.

%%%%%%%%%%%%%%%%%%%%%%%%%%%%%%%%%%%%%%%%
%%%%%%%%%%%%%%%%%%%%%%%%%%%%%%%%%%%%%%%%
 \section{Inverse Uniqueness Results in Terms of Three Discrete Spectra} \lb{s5} 
%%%%%%%%%%%%%%%%%%%%%%%%%%%%%%%%%%%%%%%%
%%%%%%%%%%%%%%%%%%%%%%%%%%%%%%%%%%%%%%%%

The findings of the preceding sections also allow us to deduce an inverse uniqueness result associated with three discrete spectra. For some relevant background references in this direction we refer to 
\cite{GS99}, \cite{HM03a}, \cite{Pi99}--\cite{Pi06}.

In order to state our result, once again, let $\tau_1$, $\tau_2$ be two Sturm--Liouville differential expressions of the form \eqref{2.2}, both satisfying the assumptions made in Hypothesis \ref{h2.1}.
 Let $S_j$ denote self-adjoint realizations of $\tau_j$, $j=1,2$, with separated boundary conditions and purely discrete spectra. Moreover, we fix some $c_j\in(a_j,b_j)$ and consider the self-adjoint restrictions $S_{a,j}$ and $S_{b,j}$ of $S_j$ to $L^2((a_j,c_j);r_j(x)dx)$ and $L^2((c_j,b_j);r_j(x)dx)$, respectively, with the separated boundary condition at $c_j$ given by 
 \begin{align}
  f(c_j)\cos(\varphi_{c,j}) - f^\qd(c_j) \sin(\varphi_{c,j}) = 0 
 \end{align}
 for some fixed $\varphi_{c,j}\in[0,\pi)$, $j=1,2$. 
In light of these preliminaries, we are now able to prove that the spectra of the three operators $S_j$, $S_{a,j}$, and $S_{b,j}$ determine the Sturm--Liouville differential expression up to a Liouville transform, provided the spectra are disjoint (see below). 
  
%%%%%%%%%%%%
 \begin{theorem}\label{thmThS}
Suppose that $S_j$ have discrete spectra and that 
\begin{equation}
\sigma(S_j) \cap \sigma(S_{a,j}) \cap \sigma(S_{b,j}) = \emptyset, \quad j=1,2.   
\end{equation}
If  
  \begin{align}\label{eqnThSequal}
  \sigma(S_{1}) = \sigma(S_{2}), \quad \sigma(S_{a,1}) = \sigma(S_{a,2}), \quad \text{and} 
  \quad  \sigma(S_{b,1}) = \sigma(S_{b,2}),
  \end{align}
 then there is a locally absolutely continuous bijection $\eta$ from $(a_1,b_1)$ onto $(a_2,b_2)$, 
 with $\eta(c_1)=c_2$, and locally absolutely continuous, real-valued functions $\kappa$, $\nu$ on 
 $(a_1,b_1)$ such that 
  \begin{align}
  \begin{split} 
 \eta' r_2\circ\eta & = \kappa^{2} r_1,   \lb{5.3} \\
  p_2\circ\eta & = \eta' \kappa^{2} p_1, \\
  \eta' s_2\circ\eta & = s_1 + \kappa^{-1}(\kappa' - \nu p_1^{-1}), \\
 \eta' q_2\circ\eta & = \kappa^2 q_1 + 2\kappa\nu s_1 - \nu^2 p_1^{-1} + \kappa'\nu - \kappa\nu'.
 \end{split} 
\end{align}
 Moreover, the map $V$ given by 
 \begin{align}\label{eqnLiouvilleTh}
 V: \begin{cases} 
 L^2((a_2,b_2);r_2(x)dx) \to L^2((a_1,b_1);r_1(x)dx), \\
 f_2 \mapsto (V f_2)(x_1) = \kappa(x_1) f_2(\eta(x_1)), \quad x_1\in(a_1,b_1),
 \end{cases} 
 \end{align}
is unitary, with 
\begin{equation} 
S_1 = V S_2 V^{-1} \quad\text{and} \quad S_{a,1}\otimes S_{b,1} = V (S_{a,2}\otimes S_{b,2}) V^{-1}. 
\end{equation}
\end{theorem}
%%%%%%%%%%%%
 \begin{proof}
  First of all, one notes that for $j=1, 2$, there are non-trivial, real entire solutions $\phi_{a,j,z}$ 
  and $\phi_{b,j,z}$ of $(\tau_j-z)u=0$ which lie in the domain of $S_j$ near $a_j$ and $b_j$, 
  respectively, $j=1, 2$. Then the spectra of the operators $S_j$,  $S_{a,j}$, and $S_{b,j}$ are the 
  zeros of the respective entire functions ($j=1,2$), 
  \begin{align}
  \begin{split} 
   W_j(z) & = \phi_{b,j,z}(c_j)\phi_{a,j,z}^\qd(c_j) - \phi_{b,j,z}^\qd(c_j) \phi_{a,j,z}(c_j), \\
   W_{a,j}(z) & = \phi_{a,j,z}(c_j)\cos(\varphi_{c,j}) - \phi_{a,j,z}^\qd(c_j)\sin(\varphi_{c,j}), \\
   W_{b,j}(z) & = \phi_{b,j,z}(c_j)\cos(\varphi_{c,j}) - \phi_{b,j,z}^\qd(c_j)\sin(\varphi_{c,j}).
   \end{split} 
  \end{align}
  Next, one introduces the auxiliary function
  \begin{align}
   N_j(z) & = \frac{W_{a,j}(z) W_{b,j}(z)}{W_j(z)} = \frac{-1}{m_{a,j}(z) + m_{b,j}(z)}, \quad z\in\C\backslash\R,
  \end{align}  
  where the meromorphic functions $m_{a,j}$ and $m_{b,j}$ are given by 
  \begin{align}
  \begin{split} 
   -m_{a,j}(z) W_{a,j}(z) & = \phi_{a,j,z}(c_j)\sin(\varphi_{c,j}) + \phi_{a,j,z}^\qd(c_j)\cos(\varphi_{c,j}), \\
   m_{b,j}(z) W_{b,j}(z) & =  \phi_{b,j,z}(c_j)\sin(\varphi_{c,j}) + \phi_{b,j,z}^\qd(c_j)\cos(\varphi_{c,j}). 
\end{split} 
  \end{align}
  Since the functions $m_{a,j}$ and $m_{b,j}$ are regular Weyl--Titchmarsh functions corresponding to $S_{a,j}$ and $S_{b,j}$ respectively, they are Nevanlinna--Herglotz functions. 
 As a result, $N_j$ is Nevanlinna--Herglotz, as well.  Moreover, based on assumption \eqref{eqnThSequal} 
and a theorem of Krein \cite[Theorem 27.2.1]{Le96}, one actually infers that $N_2 = C^2 N_1$ for some constant $C>0$.
  Consequently, the residues of $m_{b,1}$ and $m_{b,2}$ at all poles are the same up to this positive multiple $C^2$ (bearing in mind that $m_{a,1}$ and $m_{a,2}$ do not have poles there by the hypothesis of disjoint eigenvalues) and hence the corresponding spectral measures satisfy $\mu_{b,1} = C^2 \mu_{b,2}$. Applying Theorem \ref{thmdBuniqS} 
  to the operators $S_{b,1}$ and $S_{b,2}$ (upon possibly replacing $\phi_{b,1,z}$ with $C \phi_{b,1,z}$)  
yields a locally absolutely continuous bijection $\eta_b$ from $(c_1,b_1)$ onto $(c_2,b_2)$,  
  and locally absolutely continuous, real-valued functions $\kappa_b$, $\nu_b$ on $(c_1,b_1)$ such that 
  \begin{align}
  \begin{split} 
 \eta_b' r_2\circ\eta_b & = \kappa_b^{2} r_1,    \\
  p_2\circ\eta_b & = \eta_b' \kappa_b^{2} p_1, \\
  \eta_b' s_2\circ\eta_b & = s_1 + \kappa_b^{-1}(\kappa_b' - \nu_b p_1^{-1}), \\
 \eta_b' q_2\circ\eta_b & = \kappa_b^2 q_1 + 2\kappa_b\nu_b s_1 - \nu_b^2 p_1^{-1} + \kappa_b'\nu_b - \kappa_b\nu_b',
 \end{split} 
\end{align}
 on $(c_1,b_1)$. 
  Moreover, the map $V_b$ given by 
 \begin{align}
 V_b: \begin{cases} 
 L^2((c_2,b_2);r_2(x)dx) \to L^2((c_1,b_1);r_1(x)dx), \\
 f_2 \mapsto (V_b f_2)(x_1) = \kappa_b(x_1) f_2(\eta_b(x_1)), \quad x_1\in(c_1,b_1),
 \end{cases} 
 \end{align}
is unitary, with $S_{b,1} = V_b S_{b,2} V_b^{-1}$. 
  In the same way, one obtains functions $\eta_a$, $\kappa_a$ and $\nu_a$ on $(a_1,c_1)$ and a unitary mapping $V_a$ with similar properties, relating  the coefficients on the left parts of the intervals as well as the operators $S_{a,1}$ and $S_{a,2}$.

Upon joining these functions, one ends up with functions $\eta$, $\kappa$ and $\nu$ on $(a_1,b_1)$ with the properties in \eqref{5.3}. Indeed, we will now prove that these functions are absolutely continuous near $c_1$. 
  For $\eta$ this is clear because it is continuous in $c_1$ and strictly increasing. As in the proofs of Theorems \ref{thmdBuniqS} and \ref{thmTS} (cf.\ equations \eqref{eqndBdet} and \eqref{eqntheta}), one obtains for each $z\in\C$ and $x_1\in(a_1,c_1)\cup(c_1,b_1)$,
  \begin{align}\label{eqnRELtheta}
  \frac{1}{C} \begin{pmatrix} \theta_{c,1,z}(x_1) \\ \theta_{c,1,z}^\qd(x_1) \end{pmatrix} & = \begin{pmatrix} \kappa(x_1) & 0 \\ \nu(x_1) & \kappa(x_1)^{-1} \end{pmatrix} \begin{pmatrix} \theta_{c,2,z}(\eta(x_1)) \\ \theta_{c,2,z}^\qd(\eta(x_1)) \end{pmatrix},  \\
  \label{eqnRELphi} C \begin{pmatrix} \phi_{c,1,z}(x_1) \\ \phi_{c,1,z}^\qd(x_1) \end{pmatrix} & = \begin{pmatrix} \kappa(x_1) & 0 \\ \nu(x_1) & \kappa(x_1)^{-1} \end{pmatrix} \begin{pmatrix} \phi_{c,2,z}(\eta(x_1)) \\ \phi_{c,2,z}^\qd(\eta(x_1)) \end{pmatrix}, 
  \end{align}
  where $\theta_{c,j,z}$, $\phi_{c,j,z}$ is the real entire fundamental system of $(\tau_j-z)u=0$ satisfying the initial conditions
  \begin{align}
   \theta_{c,j,z}(c_j) = \phi_{c,j,z}^\qd(c_j) = \cos(\varphi_{c,j}), \quad  
   - \theta_{c,j,z}^\qd(c_j) = \phi_{c,j,z}(c_j) = \sin(\varphi_{c,j}), \quad j=1,2.
  \end{align} 
Consequently, the functions $\kappa$ and $\nu$ are absolutely continuous near $c_1$ as well. The fact that these functions satisfy the relations \eqref{5.3} is now obvious from their construction. In order to complete the proof, one notes that the transformation in \eqref{eqnLiouvilleTh} takes solutions of $\tau_2 f_2 = g_2$ to 
solutions of $\tau_1 f_1 = g_1$. Thus, one infers that $S_2$ is mapped to some self-adjoint realization, 
of $\tau_1$ with domain $V(\dom{S_2})$. However, since by construction $V = V_a\otimes V_b$ maps 
$S_{a,2}\otimes S_{b,2}$ onto $S_{a,1}\otimes S_{b,1}$, the boundary conditions of this realization 
at $a_1$ and $b_1$ (if any) are the same as those of $S_1$, concluding the proof.   
 \end{proof}
%%%%%%%%%%%%

We note that the interior boundary condition of $S_{a,2}\otimes S_{b,2}$ at $c_2$ is given by
 \begin{align}
  f(c_2) \left(\kappa(c_1)\cos(\varphi_{c,1}) - \nu(c_1) \sin(\varphi_{c,1})\right) - f^\qd(c_2) \kappa(c_1)^{-1} \sin(\varphi_{c,1}) = 0,
 \end{align}
 which is easily verified employing the relations \eqref{eqnRELtheta} and \eqref{eqnRELphi}. 

Naturally, Theorem \ref{thmTS} yields again (cf.\ Section \ref{s3}) a corresponding result for Schr\"odinger operators with distributional potentials.

%%%%%%%%%%%%
\begin{corollary}\label{corThSdp}
Suppose that $p_j= r_j=1$, that $S_j$ have discrete spectra, and that  
\begin{equation}
\sigma(S_j) \cap \sigma(S_{a,j}) \cap \sigma(S_{b,j}) = \emptyset, \quad j=1,2.   
\end{equation}
If  
  \begin{align}
  \sigma(S_{1}) = \sigma(S_{2}), \quad \sigma(S_{a,1}) = \sigma(S_{a,2}), \quad \text{and}\quad 
  \sigma(S_{b,1}) = \sigma(S_{b,2}),
  \end{align}
  then there is a locally absolutely continuous, real-valued function $\nu$ on $(a_1,b_1)$ such that 
  \begin{align}
  \begin{split} 
  s_2(c_2 - c_1 +x_1) & = s_1(x_1) + \nu(x_1), \\
  q_2(c_2 - c_1 +x_1) & = q_1(x_1) -2\nu(x_1)s_1(x_1) - \nu(x_1)^2 +\nu'(x_1),
  \end{split} 
\end{align}
 for almost all $x_1\in(a_1,b_1)$.  Moreover, the map $V$ given by 
 \begin{align}
V:  \begin{cases} 
 L^2((a_2,b_2);dx)\rightarrow L^2((a_1,b_1);dx),  \\
 f_2 \mapsto (V f_2)(x_1) = f_2(c_2-c_1 + x_1), \quad x_1\in(a_1,b_1), 
 \end{cases} 
 \end{align}
is unitary, with 
\begin{equation} 
S_1 = V S_2 V^{-1} \quad\text{and} \quad S_{a,1}\otimes S_{b,1} = V (S_{a,2}\otimes S_{b,2}) V^{-1}. 
\end{equation}  
\end{corollary}
%%%%%%%%%%%%
 
In the case of Sturm--Liouville operators in impedance form, one obtains the following result in 
analogy to Section \ref{s3}.

%%%%%%%%%%%%% 
\begin{corollary}\label{corThSimped}
Suppose that $q_j = s_j = 0$, $p_j=r_j$, that $S_{j}$ have discrete spectra, and that 
\begin{equation}
\sigma(S_j) \cap \sigma(S_{a,j}) \cap \sigma(S_{b,j}) = \emptyset, \quad j=1,2.   
\end{equation}
 If 
  \begin{align}
  \sigma(S_{1}) = \sigma(S_{2}), \quad \sigma(S_{a,1}) = \sigma(S_{a,2}), \quad \text{and}\quad 
  \sigma(S_{b,1}) = \sigma(S_{b,2}),
  \end{align}
  then there are constants $\nu_0$, $\kappa_0\in\R$ such that 
\begin{align}
p_2(c_2 - c_1 + x_1) & = p_1(x_1) \left(\nu_0 \int_{c_1}^{x_1} \frac{dt}{p_1(t)} + \kappa_0\right)^2
\end{align}
 for almost all $x_1\in(a_1,b_1)$. 
Moreover, the map $V$ given by  
\begin{align*}
V: \begin{cases} 
L^2((a_2,b_2);r_2(x)dx) \to L^2((a_1,b_1);r_1(x)dx),   \\ 
f_2 \mapsto (V f_2)(x_1) = \left(\nu_0 \int_{c_1}^{x_1} \frac{dt}{p_1(t)} + \kappa_0\right) 
f_2(c_2 - c_1 + x_1),  \quad  x_1\in(a_1,b_1), 
\end{cases} 
\end{align*}
is unitary, with 
\begin{align}
S_1 = V S_2 V^{-1} \quad\text{and} \quad S_{a,1}\otimes S_{b,1} = V (S_{a,2}\otimes S_{b,2}) V^{-1}. 
\end{align} 
\end{corollary}
 %%%%%%%%%%%%%%

 %%%%%%%%%%%%%% 
\begin{remark} We emphasize that one cannot dispense with the assumption that $S_j$ has no common eigenvalues with $S_{a,j}$ and $S_{b,j}$; \cite{GS99}. However, if one instead assumes, for example, that for each $\lambda\in\sigma(S_j)\cap\sigma(S_{a,j})\cap\sigma(S_{b,j})$
 \begin{align}
  \int_{a_1}^{c_1} |u_{1,\lambda}(x)|^2 r_1(x)dx = \int_{a_2}^{c_2} |u_{2,\lambda}(x)|^2 r_2(x)dx, 
 \end{align}
where $u_{j,\lambda}$ is some normed eigenfunction of $S_j$, then the claims of Theorem \ref{thmThS}, Corollaries \ref{corThSdp} and \ref{corThSimped} continue to hold. 
 In this case, the residues of both, $m_{a,j}$ and $m_{b,j}$ at such an eigenvalue may be determined from the residue of $m_{a,j}+m_{b,j}$ there. 
 Bearing this in mind, the proof of Theorem \ref{thmThS} (and hence also those of Corollary \ref{corThSdp} and Corollary \ref{corThSimped}) remains valid, even when $S_j$ has common eigenvalues with $S_{a,j}$ or $S_{b,j}$, $j =1,2$.  
\end{remark} 
%%%%%%%%%%%%%% 

%%%%%%%%%%%%%%%%%%%%%%%%%%%%%%%%%%%%%%%%
%%%%%%%%%%%%%%%%%%%%%%%%%%%%%%%%%%%%%%%%
 \section{Local Borg--Marchenko Uniqueness Results for Schr\"odinger Operators} \lb{s6}
%%%%%%%%%%%%%%%%%%%%%%%%%%%%%%%%%%%%%%%%
%%%%%%%%%%%%%%%%%%%%%%%%%%%%%%%%%%%%%%%%
 
In this section we will elaborate on our previous inverse uniqueness results in the case of Schr\"odinger operators with distributional coefficients. In particular, we will generalize some of the key results in 
\cite{ET12a} to the present type of operators. 
 However, before we enter a discussion of our new results, we digress for a moment and recall some 
of the classical references surrounding the Borg--Marchenko theorem and its local version. The
Borg--Marchenko uniqueness result was first published by Marchenko \cite{Ma50} in 1950, but Borg apparently had it in 1949 and it was independently published by Borg \cite {Bo52} and again by 
Marchenko \cite{Ma73} in 1952. In short, their result proved that the Weyl--Titchmarsh $m$-function 
for a Schr\"odinger operator uniquely determined the potential coefficient on $(a,b)$, assuming a regular endpoint at $a$ and a fixed boundary condition at $b$ (if any). No improvement was found until 1999 
when Simon \cite{Si99} proved a local version of this result. It roughly states that if for some 
$c \in (a, b)$, the difference of two $m$-functions is of order $\OO(\E^{-2\Im (z^{1/2}) (c-a)})$ along a ray with 
$\arg(z) = \pi-\varepsilon$ for some $\varepsilon > 0$, then the two potential coefficients coincide for 
almost all $x\in [a,a+c]$. For additional references in this context we refer to \cite{Be01}, \cite{BPW02}, 
\cite{ET12a}, \cite{Ge07}--\cite{GS00a}, \cite{KST12}, \cite{Ma05}. 

Our assumptions on the differential expression $\tau$ in the present section are contained in the following hypothesis:  
 
 %%%%%%%%%%%%
\begin{hypothesis} \lb{GHSchroe}
Suppose that $(a,b) \subseteq \bbR$, $p=r=1$ and assume that $q$, $s$ are real-valued and Lebesgue measurable on $(a,b)$ with $q$, $s\in L^1_{\loc}((a,b);dx)$.
 \end{hypothesis}
 %%%%%%%%%%%%
 
 As in the preceding sections, $S$ denotes some self-adjoint realization of $\tau$ with separated boundary conditions. Furthermore, we assume that there is a real entire fundamental system of solutions $\theta_z$, 
 $\phi_z$ as in Theorem \ref{thmREFS}, and we denote the corresponding singular 
 Weyl--Titchmarsh--Kodaira function by $m$. Now the high-energy asymptotics stated in 
 Theorem \ref{highe2} imply the following two results in a standard manner (cf.\ \cite[Lemma\ 7.1]{KST12}).

%%%%%%%%%%%%
\begin{lemma}\label{lemAsymM}
For each $x\in(a,b)$, the singular Weyl--Titchmarsh--Kodaira function $m$ and the Weyl--Titchmarsh 
solution $\psi_z$ defined in \eqref{defpsi} have the following asymptotics: 
\begin{align}\label{asymM}
m(z) &= -\frac{\theta_z(x)}{\phi_z(x)} + \OO\left(\frac{1}{\sqrt{-z}\phi_z(x)^2}\right),\\ \label{asympsi}
\psi_z(x) &= \frac{1}{2\sqrt{-z} \phi_z(x)} \left( 1 + \OO\left(\frac{1}{\sqrt{-z}}\right) \right),
\end{align}
as $|z|\to\infty$ in any sector $|\im(z)| \geq \delta\, |\re(z)|$ with $\delta>0$.
\end{lemma}
%%%%%%%%%%%%

We note that the asymptotic relation \eqref{asymM} holds for every $x\in(a,b)$ although the singular Weyl--Titchmarsh--Kodaira function on the left-hand side is obviously independent of $x$. 
 This is because the high-energy asymptotics of the quotient on the right-hand side of \eqref{asymM} is  independent of $x$ as well (up to an error term which depends on $x$). 
 More precisely, for $x_0$, $x\in(a,b)$ with $x_0 < x$ one infers that 
 \begin{align}
  \frac{\theta_z(x_0)}{\phi_z(x_0)} - \frac{\theta_z(x)}{\phi_z(x)} = \frac{1+\oo(1)}{2 \sqrt{-z} \phi_z(x_0)^2},
 \end{align}
 as $|z|\to\infty$ along nonreal rays. 
 In fact, this follows upon applying Lemma \ref{lem:asypt} and the following useful result:  

%%%%%%%%%%%%
\begin{lemma}\label{lemPhiAs}
For every $x_0$, $x\in(a,b)$ one has 
\be\label{IOphias}
\phi_z(x) = \phi_z(x_0) \E^{(x-x_0) \sqrt{-z}} \big(1 + \oo(1)\big),
\ee
as $|z|\to\infty$ along any nonreal ray.
\end{lemma}
%%%%%%%%%%%%
 
 This lemma may be verified along the lines of the proof of \cite[Lemma\ 2.4]{ET12a} employing the results of Appendix \ref{sB}. 
 As a consequence of Lemma \ref{lemPhiAs}, one also observes that the asymptotic relation \eqref{asymM} in Lemma \ref{lemAsymM} will become more and more precise as $x$ increases. 

In particular, Lemma \ref{lemAsymM} shows that asymptotics of the Weyl--Titchmarsh--Kodaira function $m$ immediately follow once one has the corresponding asymptotics for the solutions $\theta_z$ and $\phi_z$. 
 Moreover, its leading asymptotics depend only on the values of $q$ and $s$ near the left endpoint $a$ (and on the particular choice of $\theta_z$ and $\phi_z$). 
 The following local Borg--Marchenko-type uniqueness result will show that the converse is also true.
 Thus, consider Sturm--Liouville differential expressions $\tau_j$ of the form \eqref{2.2} on some intervals 
 $(a,b_j)$, $j=1,2$, respectively, satisfying the assumptions made in Hypothesis \ref{GHSchroe}. 
 By $S_j$, $j=1,2$, we denote some corresponding self-adjoint operators with separated boundary conditions.
 Furthermore, let $\theta_{j,z}$, $\phi_{j,z}$ be some real entire fundamental system of solutions as in Theorem \ref{thmREFS} and $m_j$, $j=1, 2$, be the corresponding singular Weyl--Titchmarsh--Kodaira functions.

 Before we state our local Borg--Marchenko uniqueness theorem, it is important to note that different coefficients $q$ and $s$ can give rise to the same differential expression $\tau$ (i.e., with the same 
 domain and the same action). This situation is clarified by the following lemma. 
 In order to state it, we say that $\tau_1$ and $\tau_2$ are the same on $(a,c)$ for some $c\in(a,b_1)\cap(a,b_2)$ if the differential expressions are the same when restricted to the interval $(a,c)$. 

%%%%%%%%%%%%
\begin{lemma}\label{lem:eqtau}
 For each $c\in(a,b_1)\cap(a,b_2)$ the following are equivalent:
 \begin{enumerate}[label=\,\emph{(}\roman*\emph{)},  ref=(\emph{\roman*}), leftmargin=*, widest=iii, align=left]
  \item\label{ittau1} The differential expressions $\tau_1$ and $\tau_2$ are the same on $(a,c)$.
  \item\label{ittau2} For $j,k\in \{1,2\}$, there are solutions $u_{j,k}$ of $(\tau_j-z_k) u = 0$ such that $u_{1,k} =u_{2,k}$ on $(a,c)$ and $u_{j,1}$, $u_{j,2}$ have no common zero in $(a,c)$. 
  \item\label{ittau3} The difference $s_1-s_2$ is locally absolutely continuous on the interval $(a,c)$ and $(s_1-s_2)'=  s_1^2 -s_2^2 + q_1 - q_2$ a.e.\ on $(a,c)$.
 \end{enumerate}
\end{lemma}
%%%%%%%%%%%%
\begin{proof}
For the implication $\ref{ittau1}\Longrightarrow\ref{ittau2}$ it suffices to consider two linearly independent solutions of 
$(\tau_j-z)u = 0$, $j=1,2$, for some fixed $z\in\C$. Next, assuming \ref{ittau2}, one first notes that since all 
quasi-derivatives $u_{j,k}' +s_j u_{j,k}$ are locally absolutely continuous, the difference $s_1-s_2$ is locally absolutely continuous on $(a,c)$. Using this fact to evaluate $\tau_1 u_{j,k} =\tau_2 u_{j,k}$ implies 
$(s_1-s_2)'=  s_1^2 -s_2^2 + q_1 - q_2$ a.e.\ on $(a,c)$. The final implication $\ref{ittau3}\Longrightarrow\ref{ittau1}$ 
relies on a straightforward calculation.
\end{proof}
%%%%%%%%%%%%

In order to state the next theorem, we use the short-hand notation $\phi_{1,z} \sim \phi_{2,z}$ for the asymptotic relation $\phi_{1,z}(x) = \phi_{2,z}(x) (1 + \oo(1))$  as $|z|\to\infty$ in some specified manner.
In this context we note that in view of Lemma \ref{lemPhiAs}, this holds for one $x\in(a,b_1)\cap(a,b_2)$ 
if and only if it holds for all of them. Moreover, we say the solutions $\theta_{j,z}$, $\phi_{j,z}$ are of {\it growth order at most} $\gamma$ for some $\gamma>0$ if the entire functions
 \begin{align}
 z\mapsto\theta_{j,z}(x), \quad z\mapsto\theta_{j,z}^\qd(x), \quad z\mapsto \phi_{j,z}(x), \quad 
 z\mapsto \phi_{j,z}^\qd(x), 
 \end{align}
 are of growth order at most $\gamma$ for one (and hence for all) $x\in(a,b_j)$, $j=1, 2$. 

%%%%%%%%%%%%
\begin{theorem}\label{thmbm}
Suppose that $\theta_{1,z}$, $\theta_{2,z}$, $\phi_{1,z}$, $\phi_{2,z}$ are of growth order at most $\gamma$ for some $\gamma>0$ and $\phi_{1,z} \sim \phi_{2,z}$ as $|z|\to\infty$ along some nonreal rays $R_\ell$, 
$1 \leq \ell \leq N_\gamma$, dissecting the complex plane into sectors of opening angles less than 
$\pi/\gamma$.
Then for each $c\in(a,b_1)\cap(a,b_2)$ the following properties \ref{itbm1}--\,\ref{itbm3} are equivalent:
\begin{enumerate}[label=\,\emph{(}\roman*\emph{)}, ref=(\emph{\roman*}), leftmargin=*, widest=iii, align=left]
 \item\label{itbm1} The differential expressions $\tau_1$ and $\tau_2$ are the same on the interval $(a,c)$ and $W(\phi_{1,z},\phi_{2,z})(a)=0$. 
 \item\label{itbm2} For each $\delta>0$ there is an entire function $f$ of growth order at most $\gamma$ such that
  \begin{align} 
  m_1(z)-m_2(z) = f(z) + \OO\left(\frac{1}{\sqrt{-z} \phi_{1,z}(c)^{2}}\right), 
 \end{align}
 as $|z|\rightarrow\infty$ in the sector $|\im(z)|\geq \delta\,|\re(z)|$. 
 \item\label{itbm3} For each $d\in(a,c)$ there is an entire function $f$ of growth order at most $\gamma$ such that
  \begin{align} 
  m_1(z)-m_2(z) = f(z) + \OO\left(\frac{1}{\phi_{1,z}(d)^{2}}\right),   
 \end{align}
 as $|z|\rightarrow\infty$ along the nonreal rays $R_\ell$, $1 \leq \ell \leq N_\gamma$,.
\end{enumerate}
\end{theorem}
%%%%%%%%%%%%
\begin{proof}
The implications $\ref{itbm1}\Longrightarrow\ref{itbm2}$ and $\ref{itbm2}\Longrightarrow\ref{itbm3}$ literally follow as in \cite{ET12a}. 
For the implication $\ref{itbm3}\Longrightarrow\ref{itbm1}$ one can also follow \cite{ET12a} step by step to obtain the identity
$\phi_{1,z}^2 = \phi_{2,z}^2$ on $(a,c)$. By the assumption $\phi_{1,z} \sim \phi_{2,z}$, one even 
obtains that $\phi_{1,z} = \phi_{2,z}$ on $(a,c)$ and hence the claim follows from Lemma \ref{lem:eqtau}.
\end{proof}
%%%%%%%%%%%%

As a simple consequence, one obtains the following inverse uniqueness result as in 
\cite[Corollary\ 4.3]{ET12a}.

%%%%%%%%%%%%
\begin{corollary}\label{corbm}
Suppose that $\theta_{1,z}$, $\theta_{2,z}$, $\phi_{1,z}$, $\phi_{2,z}$ are of growth order at most $\gamma$ for some $\gamma>0$ and  $\phi_{1,z} \sim \phi_{2,z}$ as $|z|\to\infty$ along some nonreal rays $R_\ell$, 
$1 \leq \ell \leq N_\gamma$, dissecting the complex plane into sectors of opening angles less than 
$\pi/\gamma$. If 
\begin{align}\label{eqncorbm}
 m_1(z) - m_2(z) = f(z), \quad z\in\C\backslash\R,
\end{align}
for some entire function $f$ of growth order at most $\gamma$, then $S_1 = S_2$.
\end{corollary}
%%%%%%%%%%%%

In the case when the operators $S_1$ and $S_2$ have purely discrete spectra with finite convergence exponent, that is,  
\begin{align}
 \inf\biggr\lbrace \omega \geq0 \,\biggr|\, \sum_{\lambda\in \sigma(S_j)} 
 \frac{1}{1+|\lambda|^{\omega}}<\infty \biggr\rbrace < \infty, \quad j = 1,\, 2,
\end{align}
it is possible to refine this result (cf.\ \cite[Corollary\ 5.1]{ET12a}).

%%%%%%%%%%%%
\begin{corollary}\label{corbmdis}
Suppose that $\phi_{1,z}$, $\phi_{2,z}$ are of growth order at most $\gamma$ for some $\gamma>0$ and 
$\phi_{1,z} \sim \phi_{2,z}$ as $|z|\to\infty$ along some nonreal rays $R_\ell$, $1 \leq \ell \leq N_\gamma$, dissecting the complex plane into sectors of opening angles less than $\pi/\gamma$.
 Furthermore, assume that $S_1$ and $S_2$ have purely discrete spectra with convergence exponent 
 at most $\gamma$.
 If 
\be
 m_1(z) - m_2(z) = f(z), \quad z\in\C\backslash\R,
\ee
for some entire function $f$, then $S_1 = S_2$.
\end{corollary}
%%%%%%%%%%%%

The lack of a growth restriction on the entire function $f$ in Corollary \ref{corbmdis} immediately 
yields a corresponding uniqueness result for the spectral measure. Closely following the proof of 
\cite[Theorem\ 5.2]{ET12a} one obtains the next result.

%%%%%%%%%%%%
\begin{theorem}\label{thmSpectFuncDisc}
Suppose that $\phi_{1,z}$, $\phi_{2,z}$ are of growth order at most $\gamma$ for some $\gamma>0$ 
and $\phi_{1,z} \sim\phi_{2,z}$ as $|z|\rightarrow\infty$ along some nonreal rays $R_\ell$, 
$1 \leq \ell \leq N_\gamma$, dissecting the complex 
plane into sectors of opening angles less than $\pi/\gamma$. Furthermore, assume that 
$S_1$ and $S_2$ have purely discrete spectra with convergence exponent at most $\gamma$. 
If the corresponding spectral measures $\mu_1$ and $\mu_2$ are equal, then $S_1 = S_2$.
\end{theorem}
%%%%%%%%%%%%

Of course, this theorem overlaps with Corollary \ref{corInvUniqSchr} to some extent. 
However, the assumptions on the real entire solutions $\phi_{j,z}$, $j=1,2$ are of a different nature. 

As another application we are also able to prove a generalization of Hochstadt--Lieberman-type uniqueness results which can be obtained along the lines of \cite[Theorem\ 5.3]{ET12a}.

%%%%%%%%%%%%
\begin{theorem}\label{thmHL}
Suppose that the operator $S_1$ has purely discrete spectrum with finite convergence exponent 
$\gamma>0$. Let $\phi_{1,z}$ and $\chi_{1,z}$ be real entire solutions of growth order at most 
$\gamma$ which lie in the domain of $S_1$ near $a$ and $b_1$, respectively, and suppose that 
there is a $c\in (a,b_1)\cap(a,b_2)$ such that
\be\label{eqeahl}
\frac{\chi_{1,z}(c)}{\phi_{1,z}(c)}=\OO(1)
\ee
as $|z|\rightarrow\infty$ along some nonreal rays $R_\ell$, 
$1 \leq \ell \leq N_\gamma$, dissecting the complex plane into sectors of 
opening angles less than $\pi/\gamma$.
 If the operator $S_2$ is isospectral to $S_1$, $\tau_1$ and $\tau_2$ are the same on $(a,c)$,  
 and $W(\phi_{1,z},\phi_{2,z})(a) = 0$, then $S_1 = S_2$.
\end{theorem}
%%%%%%%%%%%%

We note that by \eqref{IOphias}, the growth of $z\mapsto\phi_{1,z}(c)$ will increase as $c$ increases while (by reflection) the growth of $z\mapsto \chi_{1,z}(c)$
will decrease. In particular, if \eqref{eqeahl} holds for some $c\in(a,b_1)\cap(a,b_2)$ it will hold for any $c'>c$ as well.

%%%%%%%%%%%%%%%%%%%%%%%%%%%%%%%%%%%%%%%%
%%%%%%%%%%%%  The Appendix  %%%%%%%%%%%%
%%%%%%%%%%%%%%%%%%%%%%%%%%%%%%%%%%%%%%%%
%%%%%%%%%%%%%%%%%%%%%%%%%%%%%%%%%%%%%%%%

\begin{appendix}

%%%%%%%%%%%%%%%%%%%%%%%%%%%%%%%%%%%%%%%%
%%%%%%%%%%%%%%%%%%%%%%%%%%%%%%%%%%%%%%%%
\section{High-Energy Asymptotics in the General Case}    \lb{sA}
%%%%%%%%%%%%%%%%%%%%%%%%%%%%%%%%%%%%%%%%
%%%%%%%%%%%%%%%%%%%%%%%%%%%%%%%%%%%%%%%%

The aim of this appendix is to provide some results concerning high-energy asymptotics of regular  
Weyl--Titchmarsh functions. Therefore, throughout Appendix \ref{sA}, let $\tau$ be a Sturm--Liouville expression of the form \eqref{2.2}, satisfying Hypothesis \ref{h2.1}, which is regular at $a$ and $S$ be some self-adjoint realization with separated boundary conditions. 
 In this case (cf.\ Theorem \ref{thmLC}), we may choose a real entire fundamental system of solutions $\theta_z$, $\phi_z$ 
of $(\tau-z)u=0$ with the initial conditions
\begin{align}\label{phithetareg}
 \theta_z(a)=\phi_z^\qd(a)=\cos(\varphi_a), \quad -\theta_z^\qd(a)=\phi_z(a)=\sin(\varphi_a),  
 \end{align}
 for some suitable $\varphi_a\in[0,\pi)$. The corresponding Weyl--Titchmarsh function is a Nevanlinna--Herglotz function and takes on the form 
\begin{align}\label{weylmalpha}
m(z) =  \frac{\psi_z(a) \sin(\varphi_a)  + \psi_z^\qd(a)\cos(\varphi_a)}{\psi_z(a) \cos(\varphi_a)
- \psi_z^\qd(a) \sin(\varphi_a)}, \quad z\in\C\backslash\R
\end{align} 
in view of \eqref{defpsi}. 
 First we show that the asymptotic behavior of this function depends only on the left endpoint and
can be computed in terms of the asymptotics of our fundamental system of solutions.

%%%%%%%%%%%
\begin{lemma}\label{lem:asmmg}
For each $x\in(a,b)$, the Weyl--Titchmarsh function satisfies 
\be
m(z) = -\frac{\theta_z(x)}{\phi_z(x)} + \oo\left(\frac{z}{\phi_z(x)^2}\right)     \lb{A.3} 
\ee
as $|z|\to\infty$ in any sector $|\im(z)| \geq \delta\, |\re(z)|$, with $\delta>0$.
\end{lemma}
%%%%%%%%%%%
\begin{proof}
By \cite[Lemma\ 9.6]{EGNT12a} one concludes that  
\begin{align}
\begin{split} 
 \im \left( G_z(x,x) \right) & = \im(z) \int_a^b |G_z(x,y)|^2 \, r(y) dy 
= \im(z) \int_\R \left|\frac{\phi_\lam(x)}{\lam-z}\right|^2 d\mu(\lam)    \\
& = \im \int_\R \left( \frac{1}{\lam-z} - \frac{\lam}{1+\lam^2}\right) \phi_\lam(x)^2 \, d\mu(\lam), 
\quad  z\in\C\backslash\R,
\end{split} 
\end{align}
which shows that (cf.\ \cite[Corollary\ 9.8]{EGNT12a}) the diagonal of the Green's function is given by 
\begin{equation} 
G_z(x,x) = \re \left(G_i(x,x)\right) + \int_\R \left( \frac{1}{\lam-z} - \frac{\lam}{1+\lam^2}\right) \phi_\lam(x)^2 
\, d\mu(\lam),\quad z\in \C\backslash\R.
\end{equation} 
In particular, 
\begin{equation} 
G_z(x,x) = \phi_z(x) \psi_z(x)=\oo(z)
\end{equation} 
as $|z|\to\infty$ in any sector $|\im(z)| \geq \delta\, |\re(z)|$, with $\delta > 0$. Thus, dividing \eqref{defpsi} by 
$\phi_z(x)$ and solving for $m(z)$ yields \eqref{A.3}. 
\end{proof}
%%%%%%%%%%%

One notes that since $m(z)=\oo(z)$ as well as $1/m(z)=\OO(z)$ as $|z|\to\infty$ in sectors as above,
and $\phi_z(x)$ must grow super-polynomially along the imaginary axis, as all zeros lie on the real line,
the above formula indeed captures the leading asymptotics of $m$.

%%%%%%%%%%%
\begin{theorem}\label{A.2}
If $\varphi_a\ne0$, then $\int_\R (1+|\lambda|)^{-1}d\mu(\lambda)<\infty$ and 
 \begin{align}\label{eqnSThergMreg}
 m(z) = -\cot(\varphi_a) + \int_\R \frac{1}{\lambda-z} d\mu(\lambda),\quad z\in\C\backslash\R.
\end{align}
In particular, $m(z) \rightarrow -\cot(\varphi_a)$ as $|z|\rightarrow\infty$ along any non-real ray. 
Otherwise, if $\varphi_a=0$, then $\int_\R (1+|\lambda|)^{-1}d\mu(\lambda)=\infty$ 
and the above simplification is not possible.
\end{theorem}
%%%%%%%%%%%%
\begin{proof}
Since the asymptotic properties of $m$ depend only on the left endpoint $a$, one can,  without loss of generality, assume that $b$ is regular as well
and choose a Dirichlet boundary condition at this endpoint (cf.\ also Lemma\ 9.20 in \cite{Te09}).
Since by \cite[Corollary\ 10.20]{EGNT12a}  the Friedrichs  extension is associated with $\varphi_a=0$, \cite[Theorem\ 4.3]{GT} implies
$\int_\R (1+|\lambda|)^{-1}d\mu(\lambda)<\infty$ if and only if $\varphi_a\ne0$. 
Hence \eqref{eqnSThergMreg} holds with
some finite constant $c(\varphi_a)$ in place of $\cot(\varphi_a)$. Moreover, one notes that 
$c(\frac{\pi}{2})\not=0$ would imply the contradiction  $c(\varphi_a)=\infty$ for 
$\varphi_a = -\arctan(c(\frac{\pi}{2}))$ by \eqref{weylmalpha}. Hence, $c(\frac{\pi}{2})=0$ and thus
also $c(\varphi_a)= \cot(\varphi_a)$ by \eqref{weylmalpha}.
\end{proof} 
%%%%%%%%%%%

%%%%%%%%%%%%%%%%%%%%%%%%%%%%%%%%%%%%%%%%
%%%%%%%%%%%%%%%%%%%%%%%%%%%%%%%%%%%%%%%%
\section{High-Energy Asymptotics in the Special Case $p=r=1$}    \label{sB}
%%%%%%%%%%%%%%%%%%%%%%%%%%%%%%%%%%%%%%%%
%%%%%%%%%%%%%%%%%%%%%%%%%%%%%%%%%%%%%%%%

 In this appendix we will derive more precise high-energy asymptotics of regular  
 Weyl--Titchmarsh functions in the case of Schr\"odinger operators with distributional potentials 
 (i.e., under Hypothesis \ref{GHSchroe}). For Schr\"odinger operators with locally integrable potentials, 
 these facts are well-known (see, e.g., \cite[Lemma\ 9.19]{Te09}).
 
 Therefore, throughout Appendix \ref{sB}, let $\tau$ be a Sturm--Liouville differential expression of the form \eqref{2.2}, satisfying Hypothesis \ref{GHSchroe}, which is regular at $a$ and $S$ be some self-adjoint realization with separated boundary conditions. 
 As in Appendix \ref{sA}, we choose $\theta_z$, $\phi_z$ to be the real entire fundamental system of $(\tau-z)u=0$ with the initial conditions \ref{phithetareg} for some suitable $\varphi_a\in[0,\pi)$.
 
%%%%%%%%%%% 
\begin{lemma}\label{lem:asypt}
 If $\varphi_a = 0$, then for each $x\in(a,b)$,
\begin{align}
 \begin{pmatrix} \phi_z(x) \\ \phi_z^\qd(x) \end{pmatrix} &=  \frac{1}{2\sqrt{-z}}\begin{pmatrix} 1\\ \sqrt{-z} \end{pmatrix} \E^{\sqrt{-z} (x-a)} \big(1 + \oo(1)\big),    \lb{B.1} \\
 \begin{pmatrix} \theta_z(x) \\ \theta_z^\qd(x) \end{pmatrix} &= \frac{1}{2} \begin{pmatrix} 1 \\ \sqrt{-z} \end{pmatrix} \E^{\sqrt{-z} (x-a)} \big(1 + \oo(1)\big),     \lb{B.2} 
\end{align}
as $|z|\rightarrow\infty$ along any nonreal ray.
 \end{lemma}
%%%%%%%%%%%
\begin{proof}
We set $a=0$ for notational simplicity.
In this case our system \eqref{eqn:system} for the Dirichlet solution $\phi_z$ (where $\vphi_a=0$) reads
\begin{equation} 
 \begin{pmatrix} \phi_z \\ \phi_z^\qd \end{pmatrix}' 
  = \begin{pmatrix}  - s & 1 \\ q - z & s  \end{pmatrix}  \begin{pmatrix} \phi_z \\ \phi_z^\qd  \end{pmatrix}, \quad
  \begin{pmatrix} \phi_z(0) \\ \phi_z^\qd(0) \end{pmatrix} = \begin{pmatrix} 0\\ 1 \end{pmatrix}, \quad z\in\C.
\end{equation} 
Introducing for each $z\in\C\backslash\R$ the functions 
\begin{equation} 
\Phi_z(x) = \begin{pmatrix} \sqrt{-z} \phi_z(x) \E^{-\sqrt{-z} x}\\  \phi_z^\qd(x) \E^{-\sqrt{-z} x}\end{pmatrix} \quad\text{and}\quad 
E_z(x) = \E^{-2\sqrt{-z} x}, \quad x\in(0,b), 
\end{equation}
the system can equivalently be written as
\begin{align}
\Phi_z(x) &= \frac{1}{2} \begin{pmatrix} 1 - E_z(x)\\1+E_z(x) \end{pmatrix}   \\
& \quad +\frac{1}{2} \int_0^x
\begin{pmatrix} -1-E_z(x-y) & 1- E_z(x-y)\\ -1+ E_z(x-y) & 1+ E_z(x-y) \end{pmatrix} \begin{pmatrix} s(y) & 0\\ q(y)/\sqrt{-z} & s(y) \end{pmatrix} \Phi_z(y) dy.   \no
\end{align}
Consequently, $\Phi_z(x) = \Phi_\infty(x) +\oo(1)$ as $z\to\infty$ along any nonreal ray, where $\Phi_\infty$ solves
\begin{equation} 
\Phi_\infty(x) = \frac{1}{2} \begin{pmatrix} 1 \\ 1\end{pmatrix} +  \frac{1}{2} \int_0^x
  s(y) \begin{pmatrix}-1& 1\\ -1 & 1\end{pmatrix} \Phi_\infty(y) dy, \quad x\in(0,b).    \lb{B.6} 
\end{equation}
This integral equation has a unique solution which can easily be seen to be  
$\Phi_\infty(x)=\frac{1}{2} (1\;1)^\top$ upon insertion into \eqref{B.6} and noticing that 
$\frac{1}{2} (1\;1)^\top$ lies in the nullspace of $\left(\begin{smallmatrix}-1& 1\\ -1 & 1\end{smallmatrix}\right)$. 
This proves \eqref{B.1}. Relation \eqref{B.2} follows analogously.
\end{proof}
%%%%%%%%%%%%

As a simple consequence, one obtains the high-energy asymptotics of the corresponding (regular) 
Weyl--Titchmarsh functions in the case of Schr\"odinger operators with distributional potentials. 

%%%%%%%%%%%%
\begin{theorem}\label{highe2}
The Weyl--Titchmarsh function satisfies the asymptotic relation 
\begin{align}
 m(z) = \begin{cases} -\sqrt{-z} + o(z^{1/2}), & \varphi_a = 0, \\
                       -\cot(\varphi_a) + \frac{1}{\sin^{2}(\vphi_a) \sqrt{-z}} + o(z^{-1/2}), & \varphi_a\not=0, 
         \end{cases}
 \end{align}
 as $|z|\rightarrow\infty$ along any nonreal ray.
 \end{theorem}
%%%%%%%%%%%% 
\begin{proof}
The case $\vphi_a=0$ is immediate from Lemma \ref{lem:asmmg} and Lemma \ref{lem:asypt}.
The remaining case then follows from \eqref{weylmalpha}.
\end{proof}
%%%%%%%%%%%%

\end{appendix}

\medskip
\noindent

%%%%%%%%%%%
\noindent 
{\bf Acknowledgments.} We are indebted to Aleksey Kostenko for helpful hints with respect to the literature.
%%%%%%%%%%%

%%%%%%%%%%%

\end{document}